\documentclass[oneside]{amsart}
\usepackage{graphicx} 
\usepackage{amsmath, amsthm, amssymb}
\usepackage{mathtools}
\usepackage{thmtools, thm-restate}
\usepackage{hyperref}
\usepackage[english]{babel}
\usepackage{tikz-cd}
\usepackage[textsize=tiny]{todonotes}
\usepackage{float}
\usepackage{color} 
\usepackage{geometry}
\usepackage{xparse} 

\usetikzlibrary{math}
\usetikzlibrary{calc}

\usepackage[capitalize]{cleveref}
\newtheorem{lemma}{Lemma}[section]
\newtheorem{teo}[lemma]{Theorem}
\newtheorem{prop}[lemma]{Proposition}

\newtheorem{mainteo}{Theorem}
\newtheorem{maincor}[mainteo]{Corollary}
\newtheorem{mainquest}[mainteo]{Question}

\theoremstyle{definition}
\newtheorem{defn}[lemma]{Definition}
\newtheorem{notation}[lemma]{Notation}
\newtheorem{es}[lemma]{Example}
\newtheorem{oss}[lemma]{Remark}

\newtheorem{maindef}[mainteo]{Definition}

\newtheorem{claim}{Claim}[subsection]

\let\temp\phi
\let\phi\varphi
\let\varphi\temp

\let\temp\epsilon
\let\epsilon\varepsilon
\let\varepsilon\temp

\DeclareMathOperator{\sgn}{sgn}
\DeclareMathOperator{\coker}{{{{coker}}}}
\DeclareMathOperator{\im}{{{{Im}}}}
\DeclareMathOperator{\tors}{{{{tors}}}}
\DeclareMathOperator{\id}{{{{Id}}}}
\DeclareMathOperator{\supp}{supp}

\DeclareMathOperator{\vol}{Vol}
\newcommand{\str}{{{\mathrm{str}}}}

\newcommand{\fvz}{{{\operatorname{FV}_{\Z}}}}

\newcommand{\N}{\ensuremath {\mathbb{N}}}
\newcommand{\R} {\ensuremath {\mathbb{R}}}
\newcommand{\SL}{{\mathrm{SL}}}
\newcommand{\PSL}{{\mathrm{PSL}}}
\newcommand{\GL}{{\mathrm{GL}}}
\newcommand{\Z} {\ensuremath {\mathbb{Z}}}

\newcommand{\fil}{\ensuremath{\textrm{fill}}}

\NewDocumentCommand{\chains}{m m m}{{C_{{#1}}(#2;#3)}} 
\NewDocumentCommand{\bounds}{m m m}{{B_{{#1}}(#2;#3)}} 
\NewDocumentCommand{\zchains}{m m}{{C_{{#1}}(#2;\Z)}} 
\NewDocumentCommand{\zcycles}{m m}{{Z_{{#1}}(#2;\Z)}} 
\NewDocumentCommand{\zbounds}{m m}{{B_{{#1}}(#2;\Z)}} 
\NewDocumentCommand{\zhomol} {m m}{{H_{{#1}}(#2;\Z)}} 
\NewDocumentCommand{\rhomol} {m m}{{H_{{#1}}(#2;\R)}} 

\NewDocumentCommand{\norm}{m O{} O{}}{{\left\lVert#1\right\rVert_{#2}^{#3}}} 
\NewDocumentCommand{\fnorm}{m O{\Z}}{{\left\lVert#1\right\rVert_{\fil, #2}}} 
\NewDocumentCommand{\mtorus}{m m}{{{#1}_{#2}}}
\NewDocumentCommand{\hprism}{O{n} O{k} m}{P^{#1}_{#2}\left(#3\right)}
\NewDocumentCommand{\hrect}{O{n} m}{R_{#1}\left(#2\right)}
\NewDocumentCommand{\po}{m O{}}{{F^{#1}_{#2}}}

\NewDocumentCommand{\T}{O{2}}{T^{#1}}

\usepackage{scalerel}
\usepackage{upgreek}
\usepackage{stmaryrd}

\renewcommand*{\parallel}{{\stretchrel*{\sslash}{\perp}}}

\title[On the integral simplicial volume of cyclic covers of mapping tori]{On the integral simplicial volume of cyclic covers of mapping tori}

\author{Federica Bertolotti}
\address{Faculty of Mathematics, KIT, Englerstr. 2, 76131 Karlsruhe, Germany}
\email[F.~Bertolotti]{federica.bertolotti@kit.edu}

\author{Ervin Had\v{z}iosmanovi\'{c}}
\address{Scuola Normale Superiore, Piazza dei Cavalieri 7, 56126 Pisa, Italy}
\email[E.~Had\v{z}iosmanovi\'{c}]{ervin.hadziosmanovic@sns.it}

\begin{document}

\begin{abstract}
In this paper, we investigate the asymptotic behavior of the integral simplicial volume of cyclic covers of manifolds that fiber over the circle with fiber given by an $n$-dimensional torus. By studying the integral filling volume---an invariant introduced by Frigerio and the first author---for the monodromy, we establish both lower and upper bounds for the limit of the integral simplicial volume of these covers, normalized by the degree of the covering. These bounds are expressed in terms of the action of the monodromy on the real homology of the fiber.

As applications, we establish a close connection between the topological entropy and the integral filling volume of self-homeomorphisms of \( n \)-dimensional tori, we find new examples for which the $\Delta$-complexity and the integral simplicial volume are not equivalent, and we prove the nonvanishing of the filling volume for Anosov self-diffeomorphisms of infranilmanifolds.
\end{abstract}

\maketitle

\section{Introduction}

\emph{Integral simplicial volume} is a homotopy invariant for oriented closed manifolds that measures their topological complexity in terms of integral singular chains. 
For an oriented closed $n$-dimensional manifold $M$, it is defined as the minimal weighted number of singular simplices appearing in an integral cycle representing the fundamental class for $M$:
\[\norm{M}[\Z]=\min\left\{\sum_{i\in I} |a_i|\ \middle|\ \sum_{i\in I} a_i\sigma_i \text{ integral fundamental cycle for } M\right\}.\]
This invariant is the integral counterpart of the \emph{real simplicial volume} (often simply called \emph{simplicial volume}), which was introduced by Gromov in his influential paper~\cite{gromov1982volume}. 
The real simplicial volume enjoys many properties and can be studied via a dual theory known as \emph{bounded cohomology}. 
In contrast, the integral simplicial volume satisfies fewer structural properties, and no analogous cohomological theory is currently available to assist in its study. 
This makes the integral simplicial volume a more challenging and elusive invariant: except for manifolds of dimension~$1$ and~$2$, where it is completely understood, its exact value is known only for a finite number of manifolds (up to homotopy equivalence) in each dimension~\cite{loh2018odd}.
Integral simplicial volume is often investigated via its stabilized version, the \emph{stable integral simplicial volume}, defined as 
\[\norm{M}[\Z][\infty]=\inf\left\{\frac{\norm N[\Z]}{k}\ \middle|\ N\to M \text{ cover of degree }k\geq1\right\}.\]
On the contrary of integral simplicial volume, stable integral simplicial volume is stable under finite covers, meaning that if $N\to M$ is a cover of degree $k\geq 1$, then $\norm N[\Z][\infty] = k \cdot \norm M[\Z][\infty]$.

For any orientation-preserving self-homeomorphism $f \colon M \to M$ of an oriented closed manifold $M$, we consider the associated mapping torus 
\[
\mtorus{M}{f} \coloneqq M \times [0,1] / \sim, \quad \text{where } (x,0) \sim (f(x),1).
\]
In this paper, we investigate the asymptotic behavior of \( \norm{\mtorus{M}{f^k}}[\Z] \) as \( k \to \infty \), where \( \mtorus M{f^k} \to \mtorus{M}{f} \) is a cyclic covering of degree \( k \). With a slight abuse of terminology, we refer to these coverings as \emph{the cyclic coverings} of $\mtorus Mf$. Since only mapping tori are considered here, this terminology should not lead to any ambiguity.

Let us denote by $\N$ the set of positive integers, and for $n \in \mathbb{N}$ let $\T[n] = \mathbb{R}^n / \mathbb{Z}^n$ be the $n$-dimensional torus.
The main result of this paper relates the growth of the integral simplicial volume of the cyclic covers of mapping tori associated to orientation-preserving self-homeomorphisms $f\colon\T[n]\to\T[n]$ to the action induced by~$f$ on~$\rhomol{1}{\T[n]}$.

\begin{mainteo}\label{mainteo1}
	Fix $n\in \N$. There exists a constant $K_n>0$ such that the following holds: if $f\colon \T[n]\to \T[n]$ is an orientation-preserving homeomorphism on the $n$-torus  and  $f_*\colon \zhomol{1}{\T[n]}\to \zhomol{1}{\T[n]}$ denotes the induced map in homology, then
	\[\frac 1K_n \log(\rho(f_*))\le \lim\limits_{k\to+\infty}\frac{\norm{\mtorus{(\T[n])}{f^k}}[\Z]}{k} \le K_n \log(\rho(f_*)),\]
	where $\rho(f_*)$ denotes the spectral radius of the linear map $f_*$.
\end{mainteo}
See \cref{sec-matrices} for the definition of spectral radius.

To prove the upper bound of \cref{mainteo1} we exploit an invariant of orientation-preserving self-homeomorphisms of oriented closed manifolds introduced by Frigerio and the first author \cite{bertolotti2024length}. 
This invariant, called \textit{integral filling volume} and denoted by $\fvz$, was originally defined by means of the action of $f\colon M\to M$ on fundamental cycles for $M$ (see \cref{sec-maindef}). However, the two authors proved that 
\[\fvz(f)=\lim\limits_{k\to \infty}\frac{\norm{\mtorus M{f^k}}[\Z]}{k}.\]
In particular, to prove the upper bound of \cref{mainteo1} it is enough to furnish an upper bound on $\fvz(f)$, where $f\colon \T[n]\to \T[n]$ is an orientation-preserving self-homeomorphism of the $n$-torus. 
This is done via an explicit construction in \cref{upper2}.

For the lower bound we prove a more general result which can be also of interest for other manifolds (see, e.g., \cref{maincor2} below).
\begin{restatable}{mainteo}{mainteolower}\label{mainteo-lower}
    Let $n\in \N$. If $M$ is an oriented closed $n$-manifold and $f\colon M\to M$ is an orientation-preserving homeomorphism, then
    \[\fvz(f)\ge \frac 2{n(n+1)\log(n+1)}\sum_{|\lambda|>1}\log|\lambda|,\]
    where $\lambda$ varies among all the complex eigenvalues of the map $f_*\colon \rhomol{1}{M} \to \rhomol 1M$ induced by $f$ in homology.
\end{restatable}

\cref{mainteo1} completes the picture of what was already known for integral filling volume of $2$-dimensional tori and extends it to all dimensions. Up to homotopy, orientation-preserving self-homeomorphisms of $\T$ are either periodic, Dehn twists or Anosov maps, corresponding in homology to matrices in $\SL(2,\Z)$ respectively with trace less than $2$, equal to $2$ and bigger than $2$. From the definition, it follows directly that periodic maps have vanishing integral filling volume. For Dehn twists, the integral filling volume was shown to vanish by Frigerio and the first author \cite{bertolotti2024integral}. In the same paper, they showed that the integral filling volume for Anosov maps is positive and bounded from below by the logarithm of the spectral radius (up to a multiplicative constant). Thus, \cref{mainteo1} extends this lower bound to all maps and dimensions, giving a complete description of $\fvz$ on orientation-preserving self-homomorphisms of $\T[n]$ in terms of readily computable algebraic quantities.

One may ask whether it is possible to generalize \cref{mainteo1} to other manifolds, but this is not the case in general: let us consider for example a hyperbolic surface $S_g$ with genus $g\ge2$ and let $f\colon S_g\to S_g$ be a pseudo-Anosov homeomorphism. Then, the integral filling volume $\fvz(f)$ is positive \cite[Theorem 4]{bertolotti2024integral}, while there are many pseudo-Anosov homeomorphisms acting trivially in homology \cite[Corollary 14.3]{farb2011primer}, thus having $\rho(f_*)=0$.

\subsection*{Topological entropy}
Another invariant defined on self-homeomorphisms $f\colon M\to M$ of compact manifolds $M$ is the \textit{topological entropy}, denoted by $h(f)$. This invariant was introduced by Adler, Konheim and MacAndrew \cite{adler1965topological}, and it measures the complexity of the actions of maps on coverings of manifolds. As the topological entropy of a self-homeomorphism $f\colon M\to M$ is not invariant under homotopy, it makes sense to consider the \textit{minimal topological entropy}, defined on the homotopy class $[f]$ of the map $f$ by
\[h([f])\coloneqq \inf\{h(g) \mid [g]=[f]\}.\]

It is possible to state \cref{mainteo1} in terms of minimal topological entropy of the monodromy.

\begin{maincor}\label{mainteo2}
	Fix $n\in \N$. For any orientation-preserving self-homeomorphism $f\colon \T[n]\to \T[n]$ on the $n$-torus, it holds that
	\[\frac 1{nK_n} h([f])\le \fvz(f) \le K_n h([f]),\]
	where $K_n$ is the constant of \cref{mainteo1}. 
\end{maincor}

\begin{proof}
It is a classical result that for a linear homeomorphism $f\colon T^n\to T^n$ one has 
\[h(f)=\sum _{|\lambda|>1}\log|\lambda|,\]
where $\lambda$ ranges among all complex eigenvalues of the map $f_*\colon \zhomol{1}{\T[n]}\to\zhomol{1}{\T[n]}$ induced in homology by $f$, e.g, \cite[Corollary 16]{bowen1971entropy}. The quantities $\sum _{|\lambda|>1}\log|\lambda|$ and $\rho(f_*)$ are related by the following basic inequalities

\[\frac 1n\sum_{|\lambda|>1}\log|\lambda|\leq\log(\rho(f))\leq \sum_{|\lambda|>1}\log|\lambda|.\]
In particular, by \cref{mainteo1} it holds that 
\[\frac 1{nK_n} h(f)\le \fvz(f) \le K_n h(f).\]

Moreover, by a result of Manning, one has $h(g)\ge \log(\rho(g_*))$ for any homeomorphism $g\colon T^n\to T^n$ \cite{manning2006topological}. The statement then follows by combining these results and the fact that any orientation-preserving self-homeomorphism of the torus is homotopic to a linear map (see \cref{sec-matrices}).
\end{proof}

To the best of the authors’ knowledge, there are no results relating the topological entropy to the integral simplicial volume; in this sense, \cref{mainteo2} provides the first result in this direction.
Nevertheless, the simplicial volume is often studied in connection with the minimal volume entropy, an invariant of manifolds closely related to the topological entropy. In what follows, we aim to provide the reader with a brief overview of the mutual relationships among these invariants.

Let $(M,g)$ be a closed Riemannian manifold. One can consider the geodesic flow $\phi_t$ defined on the unit sphere bundle $SM$ and define the \emph{topological entropy} $h(M,g)$ of $(M,g)$ by taking the topological entropy of the flow for $t=1$ \cite[Section 3]{manning1979entropy}, that is
\[h(M,g)\coloneqq h(\phi_1).\]
To obtain an invariant of the manifold, one usually considers the \emph{topological entropy of} $M$ \cite[Sections 2.4]{kotschick2011entropies}, denoted by $h(M)$ and defined by 
\[h(M)\coloneqq \inf \{h(M,g) \mid g\text{ Riemannian metric on }M \text{ with}\vol(M,g)=1\}.\]

Let $(\widetilde{M},\widetilde{g})$ denotes the (Riemannian) universal cover of a Riemannian manifold $(M,g)$. The \emph{volume entropy} $\lambda(M,g)$ of $(M,g)$ is defined as the exponential growth rate of the volume of balls in $\widetilde{M}$ \cite[Sections 2.3]{kotschick2011entropies}, i.e.,
\[\lambda(M,g)\coloneqq \lim\limits_{R\to +\infty}\frac{\log(\vol_{\widetilde{g}}(B(p,R)))}{R},\]
where $B(p,R)\subseteq \widetilde{M}$ is the ball of radius $R$ and center $p\in\widetilde{M}$. This limit always exists and does not depend on $p$. As above, one defines the \emph{minimal volume entropy} $\lambda(M)$ of $M$ by taking the infimum of $\lambda(M,g)$ normalized by volume, that is 
\[\lambda(M)\coloneqq \inf \{\lambda(M,g) \mid g\text{ Riemannian metric on }M \text{ with}\vol(M,g)=1\}.\]

Manning proved that $\lambda(M,g)\le h(M,g)$ for any closed oriented Riemannian $n$-manifold $(M,g)$ and thus $\lambda(M)\le h(M)$ \cite[Theorem 1]{manning1979entropy}, while Gromov showed that $C(n)\norm{M}[]\le \lambda(M)^n$, where $\norm{M}[]$ is the \emph{real simplicial volume} and $C(n)$ is a constant depending only on the dimension \cite[Section 2.5]{gromov1982volume}. Thus, the following inequalities hold
\[C(n)\norm{M}[]\le \lambda(M)^n\le h(M)^n.\]

Since then, the relationship between these invariants has attracted considerable interest, e.g., \cite{gallot1991volume,babenko2021entropy}, also in the case of (4-dimensional) mapping tori \cite{bargagnati2025minimalvolumeentropymapping}.
Notice that for mapping tori with torus fiber and monodromy given by a self-homeomorphism of $T^n$ induced by a linear map the minimal topological entropy vanishes, as these manifolds admit an $\mathcal{F}$-structure \cite[Examples 19.8]{fukaya1990hausdorff}.

\subsection*{Delta-complexity of mapping tori}
Let $M$ be an oriented closed connected $n$-manifold. 
A \emph{triangulation} of $M$ is a realization of $M$ as a $\Delta$-complex, i.e., as a union of $n$-dimensional simplices with some facets identified in pairs along affine homeomorphisms. The $\Delta$-\emph{complexity} $\Delta(M)$ is defined as the minimum number of $n$-simplices that appear in a triangulation of $M$. 

The $\Delta$-complexity of mapping tori has been investigated as a standalone invariant by Lackenby and Purcell \cite{lackenby2024triangulation,lackenby2024triangulation2} and by Frigerio and the first author in relation to integral simplicial volume \cite{bertolotti2024integral}.
In particular, Frigerio and the first author showed that integral simplicial volume and $\Delta$-complexity are not bi-Lipschitz equivalent invariants on the set of cyclic covers of a mapping torus with a $2$-torus as fiber and a Dehn twist as monodromy \cite{bertolotti2024integral}. A consequence of \cref{mainteo1} is that the same holds for mapping tori with an Anosov map as monodromy.
\begin{maincor}
	There exists a sequence $\{M_i\}$ of $3$-dimensional mapping tori $M_i$ with a $2$-torus as fiber and an Anosov  map as monodromy such that 
	\[\lim\limits_{i\to \infty}\frac{\Delta(M_i)}{\norm{M_i}[\Z]}=+\infty.\]
\end{maincor}
\begin{proof}
	For any $i\in\N$, we consider the matrix $A_i\in\SL(2,\Z)$ given by
	\[A_i=
	\begin{pmatrix}
		i+1 & i\\
		1 & 1\\
	\end{pmatrix},
	\]
	and denote by $f_i\colon \T\to \T$ the map induced on the $2$-dimensional torus $\T=\R^2/\Z^2$ by $A_i$.

	Recall that the group $\PSL(2,\Z)=\SL(2,\Z)/\{\pm\id\}$ is isomorphic to the free product $\Z/2*\Z/3$, where the two factors are generated by $\overline S$ and $\overline U$ represented by the matrices \cite[Theorem 1.4]{lackenby2024triangulation}
	\[
		 S = \begin{pmatrix} 0 & -1 \\ 1 & 0 \end{pmatrix}
		\qquad
		U = \begin{pmatrix} 0 & -1 \\ 1 & -1 \end{pmatrix}.
	\]

	We denote by $\overline A_i$ the element in $ \PSL(2,\Z)$ represented by the matrix $A_i$. As $\overline A_i=(\overline U^{-1}\overline S)^{i}\overline U\overline S$ is a cyclically reduced alternating product in $\overline S,\overline U$,  it holds that $\ell(\overline A_i^j)=j(2i+2)$, where $\ell(\overline A_i)$ denotes the minimal length of a cyclically reduced word in the generators $\overline S,\overline U$.
	Then, by the work of Lackenby and Purcell \cite[Theorem 1.4]{lackenby2024triangulation}, it holds that
	\[\kappa j(2i+2)\le \Delta \left((\T)_{f_i^j}\right)\le j(i+1)+6,\]
	where $\kappa$ is a universal constant.

	On the other hand, the spectral radius $\rho(A_i)$ of the matrix $A_i$ satisfies the inequalities
	\[\log{i}<\rho(A_i)=\frac{1}{2}\left(i+2+\sqrt{i^2+4i}\right) <\log{(i+2)},\]
	Therefore, by \cref{mainteo1}, for any $i$ there exists a $k_i\in \N$ such that 
	\[\frac{k_i}{K_2}\log{(i)}\le\norm{\mtorus{(\T[2])}{f_i^{k_i}}}[\Z]\le k_iK_2\log{(i+2)}.\]
	
	The result then follows by setting $M_i\coloneqq \mtorus{(\T[2])}{f_i^{k_i}}$.
\end{proof}

\subsection*{Anosov diffeomorphisms on (infra)nilmanifolds}
Let $M$ be a compact Riemannian $n$-manifold and $f\colon M\to M$ be a diffeomorphism. The tangent bundle $TM$ comes equipped with a Riemannian metric and the differential of $f$ induces a bundle map $Df\colon TM\to TM$.
\begin{maindef}[{\cite[Section 3]{smale1967differentiable}}]
	We say that $f$ is \emph{Anosov} if there exist $Df$-invariant continuous subbundles $E^s, E^u\subseteq TM$,  constants $c>0$ and $\lambda\in (0,1)$ such that $TM=E^u\oplus E^s$ and for any $k\in \N$ it holds that
	\begin{align*}
		\norm{ Df^k(v) } &\leq c \lambda^k \| v \|, 	&\text{for any } v \in E^s, \\
		\norm{ Df^k(v) } &\geq c \lambda^{-k} \| v \|,  &\text{for any } v \in E^u.
	\end{align*}
	The subbundles $E^s$ and $E^u$ are called \emph{stable} and \emph{unstable} bundles of $f$ respectively.
\end{maindef}
The property of being Anosov does not depend on the choice of the Riemannian metric on $M$.

An automorphism $f\colon\T[n]\to \T[n]$ induced by a linear map $\widetilde f\colon \R^n\to\R^n$ with no eigenvalues on the unit circle $S^1\subseteq \mathbb{C}$ is an Anosov diffeomorphisms: indeed, the spaces $E^s, E^u$ are given by  the generalized eigenspaces corresponding to eigenvalues $\lambda$ of $\widetilde{f}$ with $|\lambda|>1$ and with $|\lambda|<1$, respectively.

A more general class of manifolds admitting Anosov diffeomorphisms are \emph{nilmanifolds}. 
\begin{maindef}
	Let $G$ be a connected, simply connected Lie group. We say that $G
	$ is \emph{nilpotent} if its Lie algebra $\mathfrak{g}$ is, that is if the lower central series
	\[\mathfrak{g}\ge[\mathfrak{g},\mathfrak{g}]\ge[\mathfrak{g},[\mathfrak{g},\mathfrak{g}]]\ge[\mathfrak{g},[\mathfrak{g},[\mathfrak{g},\mathfrak{g}]]]\ge\cdots\]
	eventually gets to the zero Lie algebra.
	A \emph{nilmanifold} is a manifold obtained as a quotient $G/\Gamma$ where $G$ is a connected, simply connected, nilpotent Lie group and $\Gamma\subseteq G$ is a discrete, finitely generated, cocompact lattice.
\end{maindef}
Connected, simply connected Lie groups admitting such a lattice are completely classified, and every finitely generated, nilpotent, torsion-free group is isomorphic to one of these lattices~\cite[Theorem~3.7]{smale1967differentiable}. 
On this class of manifolds, Anosov diffeomorphisms can be constructed algebraically, e.g., \cite[Section I.3]{smale1967differentiable}.

An even more general class of manifolds is the one of \emph{infranilmanifolds} \cite{dekimpe2012infra}. It turns out that such manifolds are finitely covered by nilmanifolds \cite[Section 2]{dekimpe2012infra}. As in the case of nilmanifolds, it is possible to construct Anosov diffeomorphisms on infranilmanifolds algebraically \cite[Sections 3, 5]{dekimpe2012infra}.
Let $M=G/\Gamma$ be a nilmanifold. Being any simply connected, nilpotent Lie group diffeomorphic to Euclidean space (via the exponential map), $M$ is aspherical and $\pi_1(M)\cong \Gamma$ is a finitely generated nilpotent group. If $M$ is an infranilmanifold it is also aspherical and $\pi_1(M)$ contains a nilpotent subgroup of finite index.

Let $f\colon M\to M$ be an orientation-preserving Anosov diffeomorphism on a infranilmanifold. By a result of Hirsch (which covers also the more general case of manifolds with virtually polycylic fundamental group) the map induced on homology $f_*\colon \rhomol{1}{M}\to\rhomol{1}{M}$ does not admit eigenvalues that are roots of unity \cite[Theorem 4]{hirsch1971anosov}. This implies that there exists at least one eigenvalue $\lambda$ with $|\lambda|>1$:
indeed, if it were not the case, all eigenvalues would be on the unit circle $S^1\subseteq \mathbb{C}$. For any eigenvalue $\lambda$, its Galois conjugates would still be eigenvalues, and thus be on the unit circle, and by Kronecker Theorem $\lambda$ would be a root of unity, e.g., \cite[Section 34, Lemma (a)]{hecke2013lectures}. 
It follows that $\rho(f_*)>1$. The following statement is then an application of \cref{mainteo2}.

\begin{maincor}\label{maincor2}
	If $M$ is an infranilmanifold and $f\colon M\to M$ an orientation-preserving Anosov diffeomorphism, then $\fvz(f)>0$.
\end{maincor}

Finitely generated nilpotent groups are residually finite and amenable. This means that fundamental groups of infranilmanifolds are amenable and residually finite. Thus, if $M$ is an infranilmanifold and $f\colon M\to M$ an Anosov diffeomorphism, the mapping torus $\mtorus{M}{f}$ is an aspherical manifold with $\pi_1(\mtorus{M}{f})\cong \pi_1(M)\rtimes_{f_*}\Z$, which is also amenable and residually finite. By a result of Frigerio, L\"{o}h, Pagliantini and Sauer, the stable integral simplicial volume $\norm{ \mtorus{M}{f}}[\Z][\infty]$  vanishes \cite{frigerio2016integral}. This gives a family of manifolds $M$ and maps $f\colon M\to M$ for which the stable integral simplicial volume $\norm{\mtorus Mf}[\Z][\infty]$ vanishes while $\fvz(f)$ does not, generalizing the examples already found by Frigerio and the first author \cite[Theorem 1.12]{bertolotti2024length}.

\subsection*{Open problems and further questions}

As previously mentioned, many maps induce the trivial map in homology while still exhibiting a positive integral filling volume. This suggests that the spectral radius considered in \cref{mainteo1} may be insufficient to fully capture this phenomenon, and naturally leads to the question of whether there exists an invariant more closely aligned with the nonvanishing of the integral filling volume.

For $2$-tori, the minimal topological entropy of homeomorphisms can serve as a meaningful invariant in this context: as shown in \cref{maincor2}, for any Anosov map $f\colon \T \to \T$, the integral filling volume $\fvz(f)$ can be bounded in terms of the minimal topological entropy $h(f)$. In this case, $h(f) = \log(\rho(f_*)) = \log(\lambda(f))$, where $\rho(f_*)$ denotes the spectral radius of the induced map $f_* \colon \rhomol{1}{\T} \to \rhomol{1}{\T}$, and $\lambda(f)$ is the stretching factor of $f$ \cite[Chapter 11]{farb2011primer}.

For pseudo-Anosov homeomorphisms \( f \colon S_g \to S_g \) on surfaces of genus \( g \geq 2 \), the integral filling volume \( \fvz(f) \) is always positive, while \( \log(\rho(f_*)) \) may vanish. The topological entropy, however, is always positive, and is given by \( h(f) = \log(\lambda(f)) \), where \( \lambda(f) \) again denotes the stretching factor \cite[Section IV]{fahti1979some}.

Moreover, in the proof of the upper bound in \cref{mainteo1}, we use the Gelfand formula for the spectral radius (see \cref{gelfand}), which implies that for any matrix \( A \in \SL(n, \Z) \) and any vector \( v \in \R^n \), the norm \( \|A^k v\| \) grows like \( \rho(A)^k \). A similar phenomenon occurs on hyperbolic surfaces: under iteration by a pseudo-Anosov homeomorphism, the lengths of simple closed curves grow like powers of the stretching factor \cite[Theorem 14.23]{farb2011primer}.

These analogies suggest that an upper bound for \( \fvz(f) \) in terms of a function of the stretching factor \( \lambda(f) \) might also exist. Similarly to what Frigerio and the first author did for the \emph{real} filling volume \cite[Question 1.11]{bertolotti2024length}, we ask the following question.

\begin{mainquest}
Let \( g > 1 \) be a positive integer. Is there a constant \( C_g \) such that for any closed oriented surface \( S_g \) of genus \( g \) and any pseudo-Anosov homeomorphism \( f \colon S_g \to S_g \), it holds that
\[
\frac{1}{C_g} \log (\lambda(f)) \leq \fvz(f) \leq C_g \log (\lambda(f)),
\]
where \( \lambda(f) \) denotes the stretching factor of \( f \)?
\end{mainquest}

At present, the only known lower bound for \( \fvz(f) \) comes from the hyperbolic volume of the associated mapping torus. However, no direct relation with \( \log(\lambda(f)) \) is known, since there exist sequences of pseudo-Anosov maps \( f_n \colon S_g \to S_g \) such that \( h(f_n) \to \infty \), while the volumes of the corresponding mapping tori remain bounded \cite{long1986hyperbolic}.

\subsection*{Plan of the paper}
In \cref{prel}, we recall the main definitions and notational conventions used throughout the paper. The expert reader may safely skip this section, except for \cref{sec-straight} and \cref{sec-prism-operator}, where some new notation is introduced. In \cref{lower_highdim}, we establish the lower bound in \cref{mainteo1}, while the upper bound is proved in \cref{upper2}.

\subsection*{Acknowledgement}
We would like to thank Roberto Frigerio for suggesting us this problem and for useful discussions about it. Moreover, we would like to thank Mikl\'{o}s Ab\'{e}rt for pointing out useful references concerning the lower bound of \cref{mainteo1}. Finally, this work has been written within the activities of GNSAGA group of INdAM (National Institute of Higher Mathematics).

\section{Preliminaries}\label{prel}

\subsection{Integral filling volume and integral simplicial volume}\label{sec-maindef}
Let $M$ be an oriented closed manifold. 
The $\ell^1$-norm on the space of singular chains $\chains{k}{M}{\Z}$ is defined by 
\[\norm{c}[1]=\norm{\sum_{i\in I}a_i\sigma_i}[1]=\sum_{i\in I}|a_i|,\]
where $c=\sum_{i\in I}a_i\sigma_i$ is an integral singular chain written as a reduced sum.

Via the $\ell^1$-norm, one can also define the \emph{integral filling norm} on the space of singular boundaries $\bounds{k}{M}{\Z}\subseteq\chains{k}{M}{\Z}$ by setting, for $b\in \bounds{k}{M}{\Z}$,
\[\fnorm{b}= \min \{\|z\|_1  \mid  z\in \zchains{k+1}M,\,\partial z=b\}.\]
Being the $\ell^1$-norm $\Z$-valued, for any $b\in\zbounds{k}{M}$ there exists a $z\in \chains{k+1}{M}{\Z}$ such that $\partial z=b$ and $\fnorm{b}=\norm{z}[1]$. 

\begin{notation}
	In the following, whenever in a statement we write the integral filling norm $\fnorm{c}$ of a chain $c$, we are tacitly including the assertion that the chain $c$ is in fact a boundary.
\end{notation}
If $n$ is the dimension of the manifold $M$, then the group $\zhomol{n}{M} \cong \Z$ is generated by the so-called \emph{integral fundamental class}, denoted by $[M]_\Z$.
A cycle $c \in \zcycles{n}{M}$ representing the integral fundamental class is called an \emph{integral fundamental cycle}.

Any continuous map $f\colon M\to N$ between manifolds induces a chain map $f_*\colon \zchains{\bullet}{M}\to \zchains \bullet N$ on singular chains, which is norm non-increasing for both the $\ell^1$-norm on chains and the filling norm on boundaries. 
If, moreover, $f\colon M\to M$ is an orientation-preserving homomorphism and $z\in \zcycles{n}{M}$ is an integral fundamental cycle, then for any $i\in\N$ the chain $f_*^i(z)$ is still an integral fundamental cycle, so that it makes sense to consider the filling norm of the boundary $z-f_*^i(z)$. We use this fact to define the integral filling volume of the map $f$.
\begin{defn}[{\cite[Definition 1.1]{bertolotti2024length}}]
	Let $f\colon M\to M$ be an orientation-preserving homotopy equivalence and denote by $f_*\colon\zchains{n}{M}\to \zchains{n}{M}$ the induced map on singular chains. Let $z\in \zcycles{n}{M}$ be an integral fundamental cycle. The \textit{integral filling volume} of $f$ is defined as
	\[\fvz(f)\coloneqq \lim\limits_{i\to\infty}\frac{\fnorm{f^i_*(c)-c}}{i}.\]
\end{defn}

The limit in this definition always exists, is independent of the choice of the integral fundamental cycle $z$, and depends only on the homotopy class of $f$ \cite[Section 2]{bertolotti2024integral}. 
ntative.

The integral filling volume can be expressed in terms of integral simplicial volume, whose definition is as follows.
\begin{defn}
	Let $M$ be an oriented closed $n$-manifold. The \textit{integral simplicial volume} of $M$ is defined as 
	\[\norm{M}[\Z]\coloneqq \min\{\|z\|_1\mid [z]=[M]_{\Z}\in\zhomol{n}{M}\}.\]
\end{defn}

Given an orientation-preserving homeomorphism $f\colon M\to M$, the \emph{mapping torus} $\mtorus Mf$ is the $(n+1)$-manifold given by the quotient
\[\mtorus Mf= M\times [0,1]/\sim\ \]
with respect to the relation $(x,0)\sim(f(x),1)$ for any $x\in M$.
In this case, the manifold $M$ is called \emph{fiber} of the mapping torus and $f$ \emph{monodromy}. 

The role of the integral filling volume lies in its close relation to the integral simplicial volume of mapping tori.
\begin{teo}[{\cite[Theorem 3]{bertolotti2024integral}}]\label{thm-fv-sv}
	Let $M$ be an oriented closed manifold and $f\colon M\to M$ be an orientation-preserving homeomorphism, then
	\begin{align*}
		\fvz(f)=\lim\limits_{k\to+\infty}\frac{\norm{\mtorus M{f^k}}[\Z]}{k}.
	\end{align*}
\end{teo}

\begin{notation}
	As in this work only integral coefficients appear, we will often omit the word ``integral''; in particular, the terms ``fundamental cycles'' and ``filling volume'' will always refer to their integral version.  
\end{notation}

\subsection{Straight simplices on the torus}\label{sec-straight}

Fix $k\in \N$ and for $i=0,\ldots,k$ let $e_i \in \R^{k +1}$ be the $i$-th vector in the canonical basis of $\R^{k+1}$.
We denote by $\Delta^k \subset \R^{k+1}$ the convex hull (with respect to the Euclidean metric) of $e_0,\ldots,e_{k}$, that is
\[\Delta^k =\left\{\left.\sum_{i = 0}^{k}\lambda_ie_i \right|0 \leq \lambda_i \leq 1, \ \sum_{i = 0}^{k}\lambda_i = 1 \right\}\subset \R^{k+1} .\]

Let us now consider another integer $n\in \N$ and let $p_0,\dots,p_{k}$ be points in $\R^n$. We define the \emph{straight simplex with vertices $p_0,\dots,p_{k}$} as the singular $k$-simplex in $\R^n$ given by the map
\[
\begin{matrix}
	\str_n(p_0,\ldots,p_{k})\colon& \Delta^k                        & \longrightarrow     & \R^n                               \\
	& \sum_{i = 0}^{k} \lambda_ie_i & \longmapsto & \sum_{i = 0}^{k} \lambda_ip_i\
\end{matrix},
\]
whose image is the (Euclidean) convex hull of the points $p_0,\ldots,p_k$.

We denote by $\T[n] = \R^n/\Z^n$ the $n$-dimensional torus (or, simply, the $n$-torus) and by $\pi_n \colon  \R^n \to T^n$ the quotient map. 
Any singular simplex in $\R^n$ projects to a simplex in $\T[n]$ via $\pi_n$. We are particularly interested in the projection of straight simplices, for which we introduce the following notation: 
\[[p_0,\ldots,p_{k}]\coloneqq \pi_n (\str_n(p_0,\ldots,p_{k}))\in C_k(T^n,\Z).\]

\begin{oss}\label{rmk: translation invariant}
	As $\T[n] =\R^n/\Z^n$, we have that for every $v\in\Z^n$ the two symbols $[p_0,\ldots,p_{k}]$ and $[p_0+v,\ldots,p_{k}+v]$ represent the same simplex in $C_k(T^n,\Z)$.
\end{oss}
\subsection{Prism operator on topological spaces}\label{sec-prism-operator}

We recall the following construction, which allows us to see the product of a simplex and an interval as a singular chain. 
Let $I=[0,1]$ and consider the space $\Delta^{k-1}\times I\subseteq \R^{k+1}=\R^{k}\times \R$. As before, denote by $e_0,\ldots,e_{k}$ the canonical basis of $\R^{k+1}$.
Then, the $k$-chain
\[p_{k}=\sum_{i=0}^{k-1} (-1)^{i}\str_{k+1}(e_0,\ldots,e_{k-i-1},e_{k-i-1}+e_{k},\ldots,e_{k-1}+e_{k})\in\zchains {k}{\Delta^{k-1}\times I}\]
is supported exactly on the product $\Delta^{k-1}\times I$.

Let now $M,N$ be two manifolds and $f,g\colon M\to N$ be continuous maps that are homotopic via the homotopy $F\colon M\times[0,1]\to N$ (i.e., $F(\cdot,0)=f$ and $F(\cdot,1)=g$). 
For any $k\in\N$, we define a \textit{prism operator} 
\[P_{k-1}\colon \zchains{k-1}{M}\to \zchains{k}{N}\]
 by sending any $(k-1)$-singular simplex $\sigma\colon\Delta^{k-1}\to M$ to the chain
\[P_{k-1}(\sigma)=(F\circ(\sigma\times \operatorname{Id}))_*(p_k),\]
and then extending the definition to the space $\zchains{k-1}{M}$ by linearity.

These prism operators satisfy the following relation
\[\partial_{k+1}\circ P_k-P_{k-1}\circ\partial_k=(-1)^k(g_*-f_*).\]
Moreover, $P_k$ has operator norm bounded from above by $k+1$, meaning that for any chain $c\in \zchains{k}{X}$, it holds that $\norm{P_k(c)}[1]\le (k+1)\norm{c}[1]$.

In the following, if no confusion arises, we drop the index $k$ in the prism operator $P_k$.
\begin{oss}
	The definition of prism operator that we gave  slightly differs from what usually is found in the literature (e.g, \cite{hatcher2005algebraic}): our prism operator coincides with the classical one when $k$ is even, while it changes sign when $k$ is odd. 
	In \cref{upper2} we will see that this definition is more convenient for our purpose, as our operator $P_n$ applied to particular kinds of fundamental cycles of the $n$-torus (which will be called \emph{parallelogram-like cycles}) yields fundamental cycles of the same kind of the $(n+1)$-torus.
\end{oss}

\subsection{Matrices and self-homotopy equivalences of tori}\label{sec-matrices}
Fix $n\in\N$.
Any matrix $A\in \SL(n,\Z)$ induces a linear map $\widetilde{f}_A\colon\R^n\to \R^n$, which descends to a homeomorphism $f_A\colon \T[n]\to \T[n]$ on the $n$-torus. In this context we say that $f_A$ is the \textit{map induced by} $A$ and that $A$ is a \emph{matrix associated to} the map $f_A$. 
Notice that if $(f_A)_*\colon \zhomol 1{\T[n]}\to\zhomol 1{\T[n]}$ denotes the map induced by $f_A$ on the first homology group of $\T[n]$, then the matrix $A$ is the one representing $(f_A)_*$ with respect to the basis of $\zhomol{1}{\T[n]}\cong\Z^n$ consisting of the homology classes represented by the $1$-cycles $[\mathbf 0,e_1],\ldots,[\mathbf 0,e_n]$ (where $\mathbf 0$ denotes the zero vector of $\R^n$). 

Since $\T[n]$ is a $K(\Z^n,1)$-space, any automorphism in $\operatorname{Aut}(\Z^n)\cong \operatorname{GL}(n,\Z)$ is induced by a  self-homotopy equivalence of $\T[n]$, which is unique up to homotopy.
In particular, the map $A\mapsto f_A$ is a noncanonical isomorphism between  $\SL(n,\Z)$ and the group of homotopy classes of orientation-preserving self-homeomorphism of the $n$-torus.

For any \( n \times n \) matrix \( A \), the \textit{spectral radius} \( \rho(A) \) is the maximum modulus of its eigenvalues:
\[\rho (A)=\max\{|\lambda|\mid \lambda\text{ is a complex eigenvalue of }A\}.\]
The spectral radius $\rho(f)$ of a linear map $f\colon \R^n\to \R^n$ is the spectral radius of any matrix associated to it. Being the eigenvalues of matrices conjugacy invariant, $\rho(f)$ is well-defined and independent of the chosen basis. 

\begin{oss}\label{basic_ineq}
For a linear map $f\colon \R^n\to \R^n$ we have the following basic inequalities:
\[\log(\rho(f))\leq \sum_{|\lambda|>1}\log|\lambda|\leq n\log(\rho(f)),\]
where $\lambda$ ranges among all the eigenvalues of $f$.
\end{oss}

We recall here the Gelfand formula, which is a classical result for the spectral radius that will become useful in the proof of the lower bound (\cref{lower_highdim}).
\begin{prop}[Gelfand formula]\label{gelfand}
	Fix $n\in \N$. Let $A$ be an $n\times n$ matrix and $\norm{\cdot}[]$ be any matrix norm. Then it holds that
	\[\rho(A)=\lim\limits_{j\to \infty}\norm{A^j}[]^{{1}/{j}}.\]
\end{prop}
We denote by \( \norm\cdot[\infty] \) the supremum norm on both the space of vectors \( \R^n \) and the space of matrices \( \R^{n \times n} \), defined for a vector \( v \in \R^n \) or a matrix \( M \in \R^{n \times n} \) as the maximum absolute value of its entries.

\section{Lower bound}\label{lower_highdim}
Let $f\colon T^2\to T^2$ be an Anosov map on the $2$-dimensional torus and $\lambda,\lambda^{-1}$ (with $|\lambda|\geq 1$) the eigenvalues of the induced map $f_*\colon\zhomol{1}{\T}\to\zhomol{1}{\T}$. Then there exists a constant $C>0$, not depending on the map $f$, for which $\fvz (f)>C\log|\lambda|$ \cite[Proposition 4.3]{bertolotti2024length}.   
In this section, we generalize this fact to all manifolds and all homeomorphisms.
\mainteolower*
The proof of \cref{mainteo-lower} relies on the following three lemmas.
\begin{lemma}\label{logtor}
    Let $n\in\N$. If $M$ is an oriented closed $n$-manifold and $f\colon M\to M$ is an orientation-preserving homeomorphism, then 
    \[\fvz(f)\ge \frac2{n(n+1)\log(n+1)}\cdot\left(\limsup\limits_{k\to\infty}\frac{\log|\tors\zhomol 1{\mtorus M{f^k}}|}{k}\right),\]
    where $\mtorus{M}{f^k}$ is the mapping torus with fiber $M$ and monodromy $f^k$.
\end{lemma}
Notice that this lemma is a generalization of a result by Frigerio and the first author for which the proof works analogously \cite[Proposition 4.3]{bertolotti2024length}.

\begin{proof}
    For $k\in\N$, let $c_k$ be an integral fundamental cycle for $\mtorus M{f^k}$ realizing the integral simplicial volume, that is $\norm {c_k}[1] = \norm {\mtorus M{f^k}}[\Z]$.
    If $s_k$ denotes the number of singular simplices involved in the reduced expression of $c_k$, then a result of Sauer \cite[Theorem 3.2]{sauer2016volume} implies that
    \[
        \log |\tors \zhomol 1 {\mtorus M{f^k}}|
        \leq \frac{n(n+1)\log(n+1)}2 \cdot s_k.
    \]
    Thus, by \cref{thm-fv-sv} we get
    \[
        \fvz(f) =
        \lim_{k\to\infty}\frac{\norm {\mtorus M{f^k}}[\Z]}k\geq
        \limsup_{k\to\infty}\frac{s_k}k\geq
        \limsup_{k\to\infty}\left(\frac 2{n(n+1)\log(n+1)}\cdot\frac{\log |\tors \zhomol 1 {\mtorus M{f^k}}|}k\right),
    \]
    as desired.
\end{proof}

Given a square matrix $M\in\GL(n,\R)$ with (complex) eigenvalues $\lambda_1,\ldots,\lambda_n$, we denote by $\widehat\det(M)$ the product of all the eigenvalues of $M$ outside the unit circle, that is
\[\widehat{\det}M=\prod_{|\lambda_i|>1}\lambda_i.\]

In view of the previous lemma, to give a lower bound for the filling volume of a map $f\colon M\to M$, one can study the growth of the torsion of the first homology groups of the mapping tori $\mtorus M{f^k}$. 
On the other hand, if $f_*\colon \zhomol 1M\to \zhomol 1M$ is the map induced on the first homology group by $f$, then a Mayer-Vietoris argument yields the isomorphism $\zhomol 1{\mtorus M{f^k}} =\coker(f^k_*-\id_G)\oplus \Z$.
This motivates the next lemma.

\begin{lemma}\label{torcoker}
    For any matrix $A\in \SL(n,\Z)$ there exists a subsequence of natural numbers $\{k_h\}$ such that
    \[\lim\limits_{h\to+\infty}\frac{\log|\tors(\coker(A^{k_h}-I))|}{k_h}= \log|\widehat \det (A)|.\]
\end{lemma}

\begin{proof}
	For a fixed $k$, let $B_k=A^k-I$.
    We compute the torsion of the cokernel of $B_k$ seen as a morphism of $\Z^n$. Let $r$ be the rank of $B_k$. 
    By the Smith normal form for integral matrices \cite[Proposition 2.11]{aluffi2009algebra} there exist matrices $P,Q\in \GL(n,\mathbb{Z})$ such that $PB_kQ$ is a diagonal $n\times n$ matrix with entries $d_1,\ldots, d_n$, where $d_i>0$  for any $1\leq i\leq r$ and $d_{r+1}=\ldots =d_n=0$.
    It follows that $\im(PB_kQ)=d_1\mathbb{Z}\oplus\dots\oplus d_r\Z\subseteq \mathbb{Z}^n$. As $P,Q$ are invertible,
    we have 
    \[
        \coker{(B_k)}\cong\coker(PB_kQ)\cong 
        \Z/d_1\Z\oplus\cdots\oplus\Z/d_r\Z\oplus\Z^{n-r},
    \]
    and so
    \[
        |\tors(\coker(B_k))|=
        \prod_{i=1}^{r}d_i.
    \]

    Assume at first that $1$ is not an eigenvalue of $A$, so that there is a strictly increasing subsequence $(k_h)_{h\in\N}$ of positive integers for which $A^{k_h}$ does not admit $1$ as eigenvalue. In particular, all eigenvalues of $B_{k_h}$ are nonzero, implying that $r=n$ and $|\det(PB_{k_h}Q)|=\prod_{i=1}^{n}d_i>0$. In this case, 
    \[
        |\tors(\coker(B_{k_h}))|=\prod_{i=1}^{n}d_i=|\det(PB_{k_h}Q)|=|\det(B_{k_h})|,
    \]
    so that \cite[Lemma, Section 2.5]{bergeron2013asymptotic}
    \[
        \lim_{h\to\infty}\frac{\log |\tors(\coker(B_{k_h}))|}{k_h}=\lim_{h\to\infty}\frac{\log|\det(A^{k_h}-\id)|}{k_h}=\widehat\det A.
    \] 

    If $1$ is an eigenvalue of $A$, then let us consider for any $k\in \N$ the operator $C_k\colon \im{(\id-A)}\to \im{(\id-A)}$ defined by 
    \[C_k(v-Av)\coloneqq v-A^kv.\]
    Then 
    \[\det(C_k)=\prod_{ \lambda\neq1}\frac{|\lambda^k-1|}{|\lambda-1|},\]
    where $\lambda$ ranges among the eigenvalues of $A$ \cite[Section 2.6]{bergeron2013asymptotic}.
    As before, we choose a strictly increasing sequence of positive integers $(k_h)_{h\in\N}$ such that $\det(C_{k_h})\neq0$, and by applying the Smith normal  form we obtain 
    \[|\tors(\coker(C_{k_h}))|=|\det(C_{k_h})|.\]
    
    Since $\im (C_{k_h})=\im(\id-A^{k_h})$, we have $\coker(C_{k_h})\subseteq\coker(\id-A^{k_h})$, so that
    \[\prod_{\lambda\neq1}\frac{|\lambda^{k_h}-1|}{|\lambda-1|}=|\tors(\coker(C_{k_h}))|\le|\tors(\coker(\id-A^{k_h}))|.\]

    By considering the map induced by $\id-A^{k_h}$ on $\Z^n/\ker(\id-A)$, one gets an upper bound \cite[Section 2.6]{bergeron2013asymptotic}
    \[|\tors(\coker(\id-A^{k_h}))|\le\prod_{\lambda\neq1} |\lambda^{k_h}-1|.\]
    Thus, we have that
    \[\sum_{\lambda\neq1}\frac{\log|\lambda^{k_h}-1|-\log|\lambda-1|}{k_h}\le \frac{\log|\tors(\coker(\id-A^{k_h}))|}{k_h}\le\sum_{\lambda\neq1} \frac{\log|\lambda^{k_h}-1|}{k_h}\]
    from which we deduce \cite[Lemma, Section 2.5]{bergeron2013asymptotic} that
        \[\lim\limits_{h\to+\infty}\frac{\log|\tors(\coker(A^{k_h}-I))|}{k_h}
        =\lim\limits_{h\to+\infty}\sum_{\lambda\neq1} \frac{\log|\lambda^{k_h}-1|}{k_h}
        = \log|\widehat \det (A)|.\]
\end{proof}
If $\zhomol 1{\mtorus M{f^k}}$ is torsion free (for example when $M=T^n$), then the last two lemmas are enough to prove \cref{mainteo-lower}.
The following lemma takes care of the torsion case.
\begin{lemma}\label{torsion}
    Let $G$ be an abelian group with torsion group $T=\tors(G)$. Let $h\colon G\to G$ be an endomorphism and $\overline{h}\colon G/T\to G/T$ the map induced by $h$ on $G/T$. Then, it holds that
    \[|\tors(\coker(\overline{h}))|\le |\tors(\coker(h))|.\]
\end{lemma}

Notice that the map $\overline h$ is well-defined as any torsion element in a group is sent to a torsion element by any endomorphism.
\begin{proof}
    Let us consider the quotient maps $\pi\colon G\to G/T$, $p\colon G\to \coker h$ and $\overline p\colon G/T\to \coker(\overline{h})$.

    Since for any $g\in G$ we have that
    \[\overline p(\pi(h(g)))=\overline{p}(\overline{h}(\pi(g)))=0,\]
    it holds the inclusion $\im(h)\subseteq \ker(\overline p\circ \pi)$.
    In particular, there is a surjective homomorphism ${\psi}\colon\coker h\to \coker\overline{h}$ making the following diagram commute:
    \begin{figure}[H]\begin{tikzcd}
        G \arrow[r, "h"] \arrow[d, "\pi", two heads] & G \arrow[d, "\pi", two heads] \arrow[r, "p", two heads] & \coker h \arrow[d, "\psi", two heads] \\
        G/T \arrow[r, "\overline{h}"]                 & G/T \arrow[r, "\overline{p}", two heads]                 & \coker \overline h.                   
    \end{tikzcd} \end{figure}
    To prove the statement, it is sufficient to show that ${\psi}$ restricts to a surjective homomorphism between the torsion subgroups, that can be proved by some diagram chasing as follows.
    
    Let $\alpha\in \tors(\coker\overline{h})$, so that $\alpha^n=1_{\coker \overline h}$ for some $n\in \N$. Since $\overline{p}\circ \pi$ is surjective, there exists $a\in G$ satisfying $\overline p ( \pi(a))=\alpha$. As $ \overline{p}(\pi(a^n))=1_{\coker\overline h}$, there exists $c\in G$ such that $\pi(a^n)=\overline{h}(\pi(c))=\pi(h(c))$, meaning that $a^n=h(c)t$ for some $t\in T$. 
    Let $m\in \N$ be such that $t^m=1$. Then, $a^{nm}=h(c^m)\in \im h$, implying that $p(a)\in \tors(\coker h)$. Moreover, by the commutativity of the above diagram, $\psi(p(a))=\alpha$. This proves the surjectivity of $\psi$ on the torsion part and concludes the proof.
\end{proof}

Now it is just a matter of putting all the ingredients  together.
\begin{proof}[Proof of \cref{mainteo-lower}]
    Let $f\colon M\to M$ be as in the statement and $f_*\colon \zhomol 1M\to \zhomol 1 M$ the isomorphism induced by $f$ on the first homology group of $M$. If $\mtorus{M}{f^k}$ is the mapping torus with fiber $M$ and monodromy $f^k$, then a Mayer-Vietoris sequence yields an  isomorphism $\zhomol 1{\mtorus M{f^k}}\cong \coker(f^k_*-\id)\oplus\mathbb{Z}$, and thus 
    \[
        |\tors(\zhomol 1{\mtorus M{f^k}})| = 
        |\tors(\coker(f^k_*-\id_G))|.
    \] 

    Let us denote by $G=\zhomol 1{M}$ the first homology group of $M$ and $T= \tors G$ the torsion subgroup of $G$. 
    As in \cref{torsion}, the map $f_*^k$ descends to a map $\overline f_*^k\colon G/T\to G/T$.
    Since the map $(f^k_*-\id_G)$ descends to the map $(\overline f_*^k-\id_{G/T})\colon G/T\to G/T$,
    by \cref{torsion}, it holds that 
    \[\left|\tors\left(\coker\left(\overline f_*^k-\id_{G/T}\right)\right)\right|\le \left|\tors\left(\coker\left(f^k_*-\id_{G}\right)\right)\right|.\]

    Since $G/T\cong\Z^d$ (for some $d\in\N$) is a finitely generated torsion-free abelian group, the map $\overline f_*$ can be represented by some matrix $A\in \SL(d,\Z)$.
    By \cref{torcoker} and \cref{logtor}, there exists a subsequence $\{k_h\}_h$ such that 
    \begin{align*}
        \left(\frac{n(n+1)\log(n+1)}2\right)^{-1}\cdot\fvz(f)
        \ge    &\limsup\limits_{h\to\infty}\frac{\log|\tors\zhomol 1{\mtorus M{f^{k_h}}}|}{k_h}\\
        =      &\limsup\limits_{h\to\infty}\frac{\log|\tors(\coker(f^{k_h}_*-\id_G))|}{k_h}\\
        \ge    &\lim\limits_{h\to\infty}\frac{\log|\tors(\coker(\overline f^{k_h}_*-\id_{G/T}))|}{k_k}\\
        =      &\lim\limits_{h\to\infty}\frac{\log|\tors(\coker(A^{k_h}-\id_{G/T}))|}{k_h}\\
        =      &\log |\widehat \det A|.
    \end{align*}
    
    The claim now follows by noticing that the matrix $A$ represents the map $\rhomol{1}{M}\to \rhomol 1M$ induced by $f$ on the first \emph{real} homology group. Indeed, 
    \[\rhomol 1M \cong G\otimes \mathbb{R}\cong (G/T\oplus T)\otimes \mathbb{R}\cong G/T\otimes \mathbb{R}\cong \mathbb{R}^d,\]
    and under these isomorphisms the map $f_*\otimes \id_\mathbb{R}$ (that is the map induced by $f$ on the first real homology group) corresponds to $\overline f_*\otimes \id_\mathbb{R}$, which can still be represented by $A$.
\end{proof}

Thanks to \cref{basic_ineq}, this also  concludes the proof of the lower bound of \cref{mainteo1}.

\section{Upper bound in dimension two}\label{upper2}
This section is devoted to proving the following theorem, which, combined with \cref{mainteo-lower}, completes the proof of \cref{mainteo1}.

\begin{teo}\label{higher-dim-upper}
    Fix $n\in \N$. There exists a constant $K_n>0$ such that the following holds: if $f\colon \T[n]\to \T[n]$ is an orientation-preserving homeomorphism on the $n$-torus  and  $f_*\colon \zhomol{1}{\T[n]}\to \zhomol{1}{\T[n]}$ denotes the induced map in homology, then
    \[\fvz(f)\le K_n\log_2(\rho(f_*)).\]
\end{teo}

To prove this theorem, we consider a fundamental cycle $c_n$ of $\T[n]$ and construct efficient fillings for $c_n-f^k_*(c_n)$ for any $k\in \N$.
In order to do that, we deal with a particular kind of fundamental cycles in $\T[n]$, which we refer as \emph{parallelogram-like} cycles and whose lifts in $\R^n$ consist in chains with supports given by Euclidean parallelograms.

In the next subsection we introduce these objects and present the key proposition (\cref{metteresuS1}) to prove \cref{higher-dim-upper}.

\subsection{Parallelogram and Rectangle-like cycles}\label{subsec-paral-rect}
Let $n\in \N$.
Given a vector $v\in \Z^n$, we consider the map $\widetilde{h}_v\colon \R^n\times [0,1]\to \R^n$ defined by $\widetilde{h}_v(\mathbf x,t)= \mathbf x+ tv$, which is a homotopy between the identity map of $\R^n$ and the translation by $v$.
Since $\widetilde{h}_v(\cdot, t)$ is a translation for every $t\in[0,1]$, the map $\widetilde h_v$ descends to a homotopy $h_v\colon \T[n]\times [0,1]\to \T[n]$ on the $n$-torus with $h_v(\cdot,0)=h_v(\cdot,1)=\id(\cdot)$.

Let
\[ \po v[\bullet]\colon \zchains \bullet{\T[n]}\to  \zchains {\bullet+1}{\T[n]}\]
be the prism operator associated to the homotopy $h_v$ (see \cref{sec-prism-operator}). Since $h_v$ is a homotopy between  $f=\id$ and $g=\id$, it holds that  $\partial \circ \po v - \po v\circ \partial=0$ and $\norm{\po v[j]}\leq j+1$.

The following lemma ensures that the operator $\po v$ preserves the space of boundaries $\zbounds \bullet{\T[n]}$ in $\T[n]$.
\begin{lemma}\label{lemma-prism-op-on-boundary}
    Let $b\in \zbounds j{\T[n]}$ be a boundary in $\T[n]$ and $v\in \Z^n$. Then $\po v(b)\in\zbounds{j+1}{\T[n]}$ is a boundary and  $\fnorm {\po v(b)}\leq (j+2)\fnorm b$.
\end{lemma}
\begin{proof}
    Let $d\in\zchains {j+1}{\T[n]}$ with $\partial d= b$ and $\norm d=\fnorm b$.
    Then 
    \[\partial (\po v(d)) = (\po v\circ \partial)(d)=\po v (b),\]
    and $\norm {\po v(d)}=(j+2)\norm d=(j+2)\fnorm b$.
\end{proof}

We are now ready to introduce \emph{parallelogram-like cycles}. They are defined inductively on the dimension, starting from dimension $1$.

\begin{defn}
    Let $\mathbf{0}\in\Z^n$ be the zero vector. The \emph{parallelogram-like $1$-cycle generated by the vector $v\in \Z^n$} in $\T[n]$  is the cycle given by
    \[\hprism[n][1]{v}\coloneqq [\mathbf{0},v]\in \zcycles{1}{\T[n]}.\]
    
    For $k>1$, having defined parallelogram-like cycles of dimension $k-1$, we define a \emph{parallelogram-like $k$-cycle $\hprism {v_1,\ldots,v_k}$ generated by the vectors $v_1,\ldots, v_k\in \Z^n$} in $\T[n]$ as the chain 
    \[\hprism {v_1,\ldots,v_k} \coloneqq \po{v_k}(\hprism[n][k-1] {v_1,\ldots,v_{k-1}})\in\zcycles{k}{\T[n]}. \]
    In this case, we call $\norm{v_1}[\infty],\ldots,\norm{v_k}[\infty]$ the \emph{sizes} of $\hprism {v_1,\ldots,v_k}$.
    
    A \emph{rectangle-like cycle with sizes $|a_1|,\ldots,|a_n|$}, 	where $a_1,\ldots, a_n\in \Z$, is a parallelogram-like cycle of the form
 	\[\hrect{a_1,\ldots,a_n}=\hprism[n][n]{a_1e_1,\ldots,a_ne_n}\in\zcycles{n}{\T[n]}.\] 
\end{defn}
 
\begin{es}
    By definition, parallelogram-like $1$-cycles are just given by straight simplices. 
    A parallelogram-like $1$-cycle of $\T[1]=S^1$ is, in particular, a rectangle-like cycle.

    An explicit expression for a $2$-dimensional parallelogram-like cycle in $\T[n]$ generated by $v,w\in \Z^n$ is given by
    \[\hprism[n][2] {v,w}\coloneqq [\mathbf 0, v,v+w] -[\mathbf{0},w,v+w]\in \zcycles2{\T[n]}\] 
    (see \cref{fig-paral-rect}).
\end{es}

\begin{center}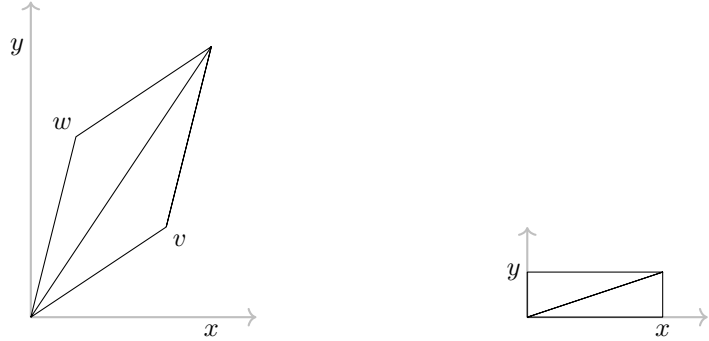
\begin{figure}[H]
  
        \tikzmath{\x=3;}
        \tikzmath{\y=1;}
        \tikzmath{\s=2;}
        \tikzmath{\t=4;}
        \tikzmath{\h=1;}
        \begin{tikzpicture}[scale=0.6]
            \draw[->, thick, color =gray!50,] (0,0) -- (\x+\y+1,0);
            \draw[->, thick, color =gray!50,] (0,0) -- (0,\s+\t+1);
            \draw[line join=round] (\x,\s) -- (\x+\y,\s+\t) -- (\y,\t) -- (0,0) -- (\x+\y,\s+\t) -- (\x,\s) -- (0,0);
            \draw[line join=round] (\x+0.3,\s-0.3) node {$v$};
            \draw[line join=round] (\y-0.3,\t+0.3) node {$w$};
            \draw[line join=round] (-0.3,\s+\t) node {$y$};
            \draw[line join=round] (\x+\y,-0.3) node {$x$};
            \begin{scope}[shift={(\x+8,0)}]
                \draw[->, thick, color =gray!50,] (0,0) -- (\x+1,0);
                \draw[->, thick, color =gray!50,] (0,0) -- (0,\y+1);
                \draw[line join=round] (\x,0) -- (\x,\y) -- (0,\y) -- (0,0) -- (\x,\y) -- (0,0) -- (\x,0);
                \draw (-0.3,\y) node {$y$};
                \draw (\x,-0.3) node {$x$};
            \end{scope}
        \end{tikzpicture}
    \caption{
       Representations of the canonical lifts in $\R^2$ of the parallelogram-like cycle $\hprism[2][2]{v,w}$ in $\T$ (on the left) and of a rectangle-like cycle (on the right).
    }\label{fig-paral-rect}
\end{figure}\end{center}

The following lemma ensures that parallelogram-like cycles are, indeed, cycles with bounded norm.
\begin{lemma}\label{lemma-paral-are-cycle}
    Every parallelogram-like cycle $\hprism {v_1,\ldots,v_k}$ is a cycle with $\norm{\hprism {v_1,\ldots,v_k}}[1]\leq k!$.
\end{lemma}
\begin{proof}
    The statement is clearly true for parallelogram-like cycles of dimension $1$. If we assume by induction that the statement is true for $k-1$, with $k>1$, then 
    \[
        \partial\hprism {v_1,\ldots,v_k}
        =\partial \po {v_k}(\hprism{v_1,\ldots, v_{k-1}})
        =\po {v_k}(\partial\hprism{v_1,\ldots, v_{k-1}})
        =0,
    \]
    and
    \[
        \norm{\hprism {v_1,\ldots,v_k}}[1]
        =\norm{\po {v_k}(\hprism{v_1,\ldots, v_{k-1}})}[1]
        \leq k\cdot\norm{\hprism{v_1,\ldots, v_{k-1}}}[1].
    \]
\end{proof}

As $\zhomol{n}{\T[n]}\cong \Z$, the previous lemma ensures that a parallelogram-like $n$-cycle in $\T[n]$ represents a multiple of the fundamental class in homology. The following lemma determines the multiple in terms of its defining vectors.

\begin{lemma}\label{paral-homology}
    The rectangle-like cycle $\hrect[n]{1,\ldots,1}$ is a fundamental cycle for $\T[n]$. More generally, if $v_1,\ldots,v_n\in \Z^n$, then the parallelogram-like cycle $\hprism[n][n]{v_1,\ldots,v_n}$ represents the element 
    \[\det A\cdot[\T[n]]_\Z\in \zhomol{n}{T^n},\]
    where $A$ is the matrix whose columns are given by $v_1,\ldots,v_n$.
\end{lemma}

\begin{proof}
    Being a cycle of degree $1$ in every point of $\T[n]$, the parallelogram-like cycle $\hprism[n][n]{e_1,\ldots,e_n}$ is a fundamental cycle for $\T[n]$.
	
    If $A$ is the matrix whose columns are given by $v_1,\ldots,v_n$, then the degree of the map $f\colon T^n\to T^n$ induced by the matrix $A$ is $\det(A)$. Thus, 
    \[\det A\cdot[\T[n]]_\Z=[f_*(\hprism[n][n]{e_1,\ldots,e_n})]=[\hprism[n][n]{v_1,\ldots,v_n}]\in \zhomol{n}{T^n},\]
    as desired.
\end{proof}

A direct consequence of the previous lemma is the fact that for any $v_1,\ldots,v_n\in\Z^n$, the parallelogram-like cycle $\hprism[n][n]{v_1,\ldots,v_n}$ and the rectangle-like cycle $\hrect {\det(A),1,\ldots,1}$ are cobordant, where $A$ is the matrix with columns $v_1,\ldots,v_n$. 
The key step for the proof of \cref{higher-dim-upper} is to bound the filling norm of their difference, as stated in the following proposition.

\begin{restatable}{prop}{propupperbound}\label{metteresuS1}
    Let $n\in \N$. There exists a constant $K_n>0$ such that for any $v_1,\ldots,v_n\in\Z^n$ it holds that
    \[
        \fnorm{\hprism[n][n]{v_1,\ldots,v_n} - \hrect {\det(A),1,\ldots,1}}
        \leq K_n\log_2(\norm{A}[\infty])+K_n, \]
    where $A$ is a matrix whose columns are given by $v_1,\ldots,v_n$.
\end{restatable}

Recall that $c_n=\hrect{1,\ldots,1}$ is a fundamental cycle for $\T[n]$ (\cref{paral-homology}).
If $A\in \SL(n,\Z)$ and $f\colon \T[n]\to \T[n]$ is the map induced by $A$, then $P_k=f^k_*(c_n)$ is a parallelogram-like fundamental cycle. The previous proposition furnishes an estimate for $\fnorm{P_k-R}$, where $R$ is the rectangle-like cycle given by $R=\hrect {\det A,1,\ldots,1}=c_n$. This estimate leads to a proof of \cref{higher-dim-upper}.

\begin{proof}[Proof of \cref{higher-dim-upper} assuming \cref{metteresuS1}]
    As $\fvz$ is homotopy invariant, we can assume that $f\colon \T[n]\to \T[n]$ is a map induced by a matrix $A\in \SL(n,\Z)$. Let us denote by $\rho(A)$ the spectral radius of $A$.

    By the Gelfand formula for the spectral radius (\cref{gelfand}), one has
    \[\rho(A)=\lim\limits_{j\to \infty}\left(\norm{A^j}[\infty]\right)^{{1}/{j}}.\]
    For any $\epsilon>0$, there exists an $M>0$ (depending on $\epsilon$) such that for any $j>M$ it holds that 
    \[\norm{A^j}[\infty]\leq (1+\epsilon)^j\rho(A)^j.\]    

    Fix $c_n=\hrect[n]{e_1,\ldots,e_n}$ a fundamental cycle for $T^n$ (\cref{paral-homology}).
    For $j\in\N$ we set $v_i^j=A^je_i$, so that $f_*^j(c_n)=\hprism[n][n]{v^j_1,\ldots,v^j_n}$. As $A\in \SL(n,\Z)$ we have that 
    \[\hrect {\det(A^j),1,\ldots,1}=\hrect{1,\ldots,1}=c_n.\]
    
    By \cref{metteresuS1} and the definition of filling volume, one then obtains 
    \begin{align*}
        \fvz(f)
        =&\lim\limits_{j\to \infty}\frac{\fnorm{f^j_*(c_n)-c_n}}{j}\\
        =&\lim\limits_{j\to \infty}\frac{\fnorm{\hprism[n][n]{v^j_1,\ldots,v^j_n}-\hrect {\det(A^j),1,\ldots,1}}}{j}\\
        \leq& \lim_{j\to\infty}\frac{K_n\log_2(\norm{A^j}[\infty])+K_n}{j}\\
        \leq& \lim_{j\to\infty}\frac{K_n\log_2((1+\epsilon)^j\rho(A)^j)+K_n}{j}\\
        =&K_n \log_2 ((1+\epsilon)\rho(A)).
    \end{align*}
    The claim then follows from the arbitrariness of $\epsilon$.
\end{proof}

The rest of this section is devoted to the proof of \cref{metteresuS1}.
In order to understand the idea of the proof, we consider the graph whose vertices are given by integral cycles in $\T[n]$, and where two vertices are joined by an edge if and only if the corresponding cycles $a,b$ are cobordant and satisfy $\fnorm{a-b}=1$. In other words, this graph can be seen as a metric space whose distance function is induced by the filling norm. The final goal is to estimate the distance between $\hprism[n][n]{v_1,\ldots,v_n}$ and $\hrect {\det(A),1,\ldots,1}$. 

We proceed as follows:
\cref{sec-splitting} consists in proving that we can split a parallelogram-like cycle into two cycles without going too far in our graph. To be more precise, we prove that for any $v_1',v_1'',v_2,\ldots, v_k\in\Z^n$ the sum $\hprism{v_1',v_2,\ldots,v_k}+\hprism{v_1'',v_2,\ldots,v_k}$ is close to the parallelogram-like cycle $\hprism{v_1'+v_1'',v_2,\ldots,v_k}$ (\cref{lemma-splitting-vector}).

In \cref{sec-slim} we deal with ``slim'' parallelogram-like cycles, which are the parallelogram-like cycles whose generating vectors are linearly dependent. In particular, we  control the distance between these cycles and the zero cycle in terms of their sizes (\cref{lemma-slim-parallelogram-hdim}).

\cref{sec-rect} deals with rectangle-like cycles, proving that each of them is close to a rectangle-like cycle with all but one size equal to $1$ (\cref{lemma-rect-with-size1-hdim}).
Moreover, we show that any parallelogram-like cycle is close to a sum of rectangle-like cycles with a controlled number of summands (\cref{lemma-hprism-sum-rect}).

In \cref{sec-proof-metteresuS1} we put together all these results and get \cref{metteresuS1}.

\begin{oss}
    To simplify the notation, we denote all the constants appearing in the following lemmas by $C_n$, as they can be assumed to coincide up to taking the maximum. In the specific case in which we need to use the constant of a statement inside its proof (for example, when the proof is done by induction), then we will refer to such a constant as $V_n$.
    
    Here, the subscript $n$ in $C_n$ and $V_n$ denotes the dependence on the dimension. 
\end{oss}

We begin with the following lemma, which bounds the distance between two parallelogram-like cycles with the same generating vectors but arranged in different orders. 
\begin{lemma}\label{lemma-rearranging-vectors}
	Let $k,n\in \N$ be two positive integers. There exists a constant $C_k$, depending only on $k$, such that for every $v_1,\ldots, v_k\in \Z^n$ and every permutation $\theta$ on $\{1,\ldots,k\}$ it holds that  
	\[\fnorm{\hprism{v_1,\ldots,v_k}-\sgn{\theta}\cdot\hprism{v_{\theta(1)},\ldots,v_{\theta(k)}}}\leq C_k\]
    and
    \[\fnorm{\hprism{v_{1},\ldots,v_{k}}+\hprism{-v_1,v_2,\ldots,v_k}}\leq C_k.\]
\end{lemma}
Recall that the sign $\sgn \theta$ of a permutation $\theta\colon \{1,\ldots,k\}\to \{1,\ldots, k\}$ is defined to be $+1$ if the permutation is even and $-1$ if the permutation is odd.
\begin{proof}
	Denote by $e_1,\ldots,e_k$ the canonical basis for $\R^k$. Let us set  $c= \hprism[k][k]{e_1,\ldots,e_k}, c_-= \hprism[k][k]{-e_1,e_2,\ldots,e_k}$, and $c_\sigma= \hprism[k][k]{e_{\sigma(1)},\ldots,e_{\sigma(k)}}$ for any permutation $\sigma\colon \{1,\ldots,k\}\to \{1,\ldots, k\}$. 
    Let $A_\sigma$ (resp. $A_-$) be the matrix whose columns are given by ${e_{\sigma(1)},\ldots,e_{\sigma(k)}}$ (resp. $-e_1,e_2,\ldots,e_k$). Clearly, $\det A_-=-1$.
    Moreover, being $A_\sigma$ the permutation matrix corresponding to $\sigma$, we have $\det A=\sgn \sigma$.

	By \cref{paral-homology}, $c, -c_-$, and $\sgn \sigma \cdot c_\sigma$ represent the same class in homology. We set 
    \[C_k'=\max\big\{\fnorm{c-\sgn \sigma \cdot c_\sigma}\ \big|\ \sigma\in \mathrm{Sym}(\{1,\ldots,k\})\big\},\]
	and $C_k=\max\{C_k',\fnorm{c+c_-}\}$.

    Let $f\colon \T[k]\to\T[n]$ be the map induced by the linear map $\R^k\to \R^n$ sending $e_i$ to $v_i$ for any $i=1,\ldots, k$. 
    Let us fix a permutation $\theta$ on $\{1,\ldots,k\}$. 
    As $\hprism[n][k]{v_1,\ldots,v_k}=f_*(c), \hprism[n][k]{-v_1,v_2,\ldots,v_k}=f_*(c_-)$ and $\hprism{v_{\theta(1)},\ldots,v_{\theta(k)}}=f_*(c_\theta)$, it holds that
	\begin{align*}
	    \fnorm{\hprism{v_{\theta(1)},\ldots,v_{\theta(k)}}-\sgn{\theta}\cdot \hprism{v_1,\ldots,v_k}}	
        =\fnorm{f_*(c-\sgn \theta \cdot c_\theta)}\leq C_k,
	\end{align*}
    and 
    \[
        \fnorm{\hprism{v_{1},\ldots,v_{k}}+\hprism{-v_1,v_2,\ldots,v_k}}
        =\fnorm{f_*(c+  c_-)}
        \leq C_k.
    \]
\end{proof}

\subsection{Splitting of parallelogram-like cycles}\label{sec-splitting}
As mentioned above, a parallelogram-like cycle can be split into the sum of two parallelogram-like cycles as in the following lemma.
\begin{lemma}\label{lemma-splitting-vector}
    Let $k\in \N$. There exists a constant $C_k$ such that  for any $v_1,\ldots,v_k, v_i',v_i''\in \Z^n$ with $v_i=v_i'+v_i''\in \Z^n$ it holds that
    \[
        \fnorm{
            \hprism{v_1,\ldots,v_k}-\hprism{v_1,\ldots,v_i',\ldots,v_k}-\hprism{v_1,\ldots,v_i'',\ldots,v_k}
        }\leq C_k.
    \]
\end{lemma}

Before proving this statement, we need the following claim, which bounds the filling norm of a parallelogram-like cycle that has $\mathbf{0}$ among its generating vectors.

\begin{claim}\label{lemma-zero-vector}
    There exists a constant $C_k$, depending on $k\in\N$, such that for any $v_1,\ldots,v_{k}\in \Z^n$ it holds that
    \[
        \fnorm{
            \hprism[n][k+1]{v_1,\ldots,v_{k},\mathbf 0}
        }\leq C_k.
    \]
\end{claim}
\begin{proof}
    We denote by $\{e_1,\ldots, e_{k}\}$ the canonical basis of $\R^{k}$. 
    The parallelogram-like cycle $q_k=\hprism[k][k+1]{e_1,\ldots, e_{k},\mathbf 0}$ represents the trivial class in $\zhomol{k+1}{\T[k]}=0$.
    Set $C_k=\fnorm{q_k}$ and consider the map $\nu\colon \T[k]\to\T[n]$ induced by the linear map $\R^k\to\R^n$ sending $e_i$ to $v_i$ for every $i=1,\ldots,k$. Then, we obtain
    \[\fnorm{\hprism[n][k+1]{v_1,\ldots,v_{k},\mathbf 0}}= \fnorm{\nu_*(q_k)}\leq C_k. \]
\end{proof}

We also need to fix a canonical way to lift parallelogram-like cycles to $\R^n$; we do it as follows. 
If $c=\sum a_i\sigma_i\in \zchains k {\R^n}$ is written in reduced form, then the support of $c$ is 
\[\supp c= \bigcup_i \im \sigma_i.\]
With this notation, a lift $P\in\zchains k {\R^n}$ of $\hprism{v_1,\ldots,v_k}$ to $\R^n$ is \emph{canonical} if 
\[\supp P\subset \left\{\sum_{i=1}^k \epsilon_iv_i\ | \ \epsilon_i\in [0,1]\right\}, \] 
i.e., if the support of $P$ is contained in the Euclidean parallelogram with coordinate vectors $v_1,\ldots,v_k$.
Such a lift always exists and is unique.

\begin{oss}\label{rmk-boundary-lift-prism}
    The boundary of the canonical lift $P$ of a parallelogram-like cycle $\hprism{v_1,\ldots,v_k}$  is a sum 
    \[\partial P=\sum_{i=1}^k (-1)^{i}(b_i^+-b_i^-),\]
    where for every $i=1,\ldots, k$ one has that
    \begin{itemize}
        \item $b_i^+\in\zchains{k-1}{\R^n}$ is the canonical lift of the parallelogram-like cycle $\hprism{v_1,\ldots,\widehat{v}_i,\ldots, v_k}$;
        \item $b_i^-=(\tau_{v_i})_*(b_i^+)$, where $\tau_{v_i}\colon \R^n\to \R^n$ is the translation by $v_i$ in $\R^n$.
    \end{itemize}
    In particular, $(\pi_n)_*(b_i^+-b_i^-)=0$, where $\pi_n\colon \R^n\to \T[n]$ denotes the quotient map.
\end{oss}

\begin{proof}[Proof of \cref{lemma-splitting-vector}]
    Let $\pi\colon \R^n\to \T[n]$ be the quotient map and denote by $\{e_1,\ldots, e_{n+1}\}$ the canonical basis of $\R^{n+1}$. 
    Up to applying \cref{lemma-rearranging-vectors}, we can assume $i=k$. Let  $v_1,\ldots, v_k,v_k',v_k'' \in \Z^n$ be as in the statement. 

     We are going to define a chain $Q\in \zchains {k+1}{\T[n]}$ (represented in \cref{fig-slide-paral}) with $\norm {Q}[1]\leq (k+1)!$  and
	\begin{multline*}
		\partial Q=(-1)^k\big(-\hprism{v_i,\ldots,v_{k-1},v_k'}-\hprism{v_1,\ldots,v_{k-1},v_k''}+\\
        +\hprism{v_i,\ldots,v_{k-1},v_k'+v_k''}+\hprism{v_1,\ldots,v_{k-1},\mathbf 0}\big).
	\end{multline*}

    Let $\iota\colon \R^n\to \R^{n+1}$ be the inclusion into the first $n$ factors and  $p\colon \R^{n+1}\to \R^{n}$ be the projection onto the first $n$ factors. Set $w_k'=\iota(v_k')+e_{n+1}$.

    We denote by $P$ the canonical lift of the parallelogram-like cycle $\hprism[n+1]{\iota(v_1),\ldots,\iota(v_{k-1}), w_k'}$ and by $P'$ the one of $\hprism[n+1]{\iota(v_1),\ldots,\iota(v_{k-1}),w_{k}'+\iota(v_k'')}$.

    Consider the prism operator $R$ associated to the homotopy $f\colon\R^{n+1}\times [0,1]\to \R^{n+1}$ that at time $t$ is represented by the matrix
    \[
        f_t=f(\cdot,t)=
        \begin{pmatrix}
            I_n     &t \cdot v_k''\\
            0       &1
        \end{pmatrix}.
    \]
    We set 
    \[ Q = \pi_*(p_*(R(P))).\]

    Since $f_0=\id$, $f_1(\iota(v_i))=\iota(v_i)$ for $i=1,\ldots,k-1$ and $f_1(w'_k)=w'_k+\iota(v''_k)$ we have that $(f_0)_*(P)=P$ and $(f_1)_*(P)=P'$. By the properties of the prism operator  (\cref{sec-prism-operator}) it holds that
    \begin{align*}
        \partial Q
        &=\pi_*(p_*(\partial(R(P))))\\
        &=\pi_*(p_*(R(\partial P)))+(-1)^k(\pi_*(p_*(P'))-\pi_*(p_*(P)))\\
        &= \pi_*(p_*(R(\partial P)))+(-1)^k(\hprism{v_i,\ldots,v_{k-1},v_k'+v_k''}-\hprism{v_i,\ldots,v_{k-1},v_k'}).
    \end{align*}

    By \cref{rmk-boundary-lift-prism} we have that $\partial P = \sum_{i=1}^{k}(-1)^i(b_i^+-b_i^-)$, where $b_k^-=(\tau_{w_k'})_*(b_k^+)$ and $b_i^-=(\tau_{\iota(v_i)})_*(b_i^+)$ for every $i=i,\ldots,k-1$ (with $\tau_v\colon\R^{n+1}\to\R^{n+1}$ denoting the translation by $v$). 
    Since for any $i=1,\ldots,k-1$ the translation $\tau_{\iota(v_i)}$ commutes with $f_t$, the induced map $(\tau_{\iota(v_i)})_*$ on the chain complex  commutes with the operator  $R$ and thus
    \[
        \pi_*\big(p_*(R(b_i^-))\big)=\pi_*\big(p_*(R((\tau_{\iota(v_i)})_*(b_i^+)))\big)=
        \pi_*(p_*((\tau_{v_i})_*(R(b_i^+))))=\pi_*(p_*(R(b_i^+))).
    \]
	Thus, all terms in the sum $\pi_*(p_*(R(\partial P)))= \sum_{i=1}^{k}(-1)^i\pi_*(p_*(R(b_i^+-b_i^-)))$ vanish except possibly the one corresponding to the index $i=k$.
    In this case, since $b_k^+$ is the canonical lift of $\hprism[n+1][k-1]{\iota(v_1),\ldots,\iota(v_{k-1})}$, we have that $(f_t)_*(b_k^+)=b_k^+$. In particular, the homotopy $f$ acts on $b_k^+$ as $h_\mathbf{0}$ (defined in \cref{subsec-paral-rect}). Thus, 
    \[ 
        \pi_*(p_*(R(b_k^+)))=\po{\mathbf{0}}(\pi_*(p_*(b_k^+)))=
        \hprism{v_1,\ldots,v_{k-1},\mathbf 0}.
    \]

    Finally, since $b_k^-=(\tau_{w_k'})_*(b_k^+)$ is contained in the affine $(n-1)$-subspace $\mathcal H= \iota(\R^n)+w'_k$ and $f_t|_{\mathcal H}=\tau_{t\cdot\iota(v_k'')}|_{\mathcal{H}}$, we have that $(f_t)_*(b_k^-)=(\tau_{t\cdot\iota(v_k'')})_*(b_k^-)=(h_{t\cdot\iota(v_k'')})_*(b_k^-)$, getting
    \[ 
        \pi_*(p_*(R(b_k^-)))=\po{v_k''}(\hprism[n][k-1]{v_1,\ldots,v_{k-1}})=\hprism{v_1,\ldots,v_{k-1},v_k''}.
    \]

    Putting everything together, we get
    \begin{align*}
        \pi_*(p_*(R(\partial P)))
        = (-1)^k(\hprism{v_1,\ldots,v_{k-1},\mathbf 0}-\hprism{v_1,\ldots,v_{k-1},v_k''})
    \end{align*}
	and, therefore, 
	\begin{multline*}
		\partial Q=+(-1)^k\big(-\hprism{v_i,\ldots,v_{k-1},v_k'}-\hprism{v_1,\ldots,v_{k-1},v_k''}+\\
        +\hprism{v_i,\ldots,v_{k-1},v_k'+v_k''}+\hprism{v_1,\ldots,v_{k-1},\mathbf 0}\big).
	\end{multline*}
    
    As $v_k=v_k'+v_k''$, it follows that
    \begin{multline*}
        \fnorm{\hprism{v_1,\ldots,v_{k-1},v_k}-\hprism{v_1,\ldots,v_{k-1},v_k''}-\hprism{v_i,\ldots,v_{k-1},v_k'}}\\
        \leq \fnorm{\partial Q}+\fnorm{\hprism{v_1,\ldots,v_{k-1},\mathbf 0}}.
    \end{multline*}
    The desired result then follows from the fact that $\fnorm{\hprism{v_1,\ldots,v_{k-1},\mathbf 0}}\leq C_{k-1}$, by \cref{lemma-zero-vector}, and
    \begin{multline*}
        \fnorm{\partial Q}
        \leq\norm Q[1] 
        = \norm{\pi_*(p_*(R(P)))}[1]
        \leq (k+1)\norm{\hprism[n+1]{\iota(v_1),\ldots,\iota(v_{k-1}), w_k'}}[1]
        \leq (k+1)!,
    \end{multline*}
    by \cref{lemma-paral-are-cycle}. 
\end{proof}

\begin{center}\begin{figure}[H]
    \begin{minipage}{0.4\textwidth}
\centering
\tikzmath{\vx=2;}
\tikzmath{\vy=0;}
\tikzmath{\wx=-1;}
\tikzmath{\wy=5;}
\tikzmath{\ux=\wx+4;}
\tikzmath{\uy=\wy-2;}
\tikzmath{\h=1;}
\begin{tikzpicture}[scale=0.6]
	\draw[line join=round,dashed] (\vx,\vy) -- (\vx+\wx,\vy+\wy);
	\draw[line join=round,dashed] (0,0) -- (\vx+\wx,\vy+\wy);
	\draw[line join=round] (\wx,\wy) -- (\vx+\wx,\vy+\wy);
	\draw[line join=round] (\vx+\ux,\vy+\uy,\h) -- (\vx+\wx,\vy+\wy) -- (\ux,\uy,\h);
	\draw[line join=round] (\vx,\vy) -- (\vx+\ux,\vy+\uy,\h) -- (\ux,\uy,\h) -- (0,0) -- (\vx,\vy);
	\draw[line join=round] (\vx+\ux,\vy+\uy,\h) -- (0,0);
	\draw (0,0) -- (\wx,\wy) -- (\ux,\uy,\h);
	\draw[line join=round, red, thick] (\vx,\vy) -- (0,0);
	\draw (0,0) -- (\wx,\wy) -- (\ux,\uy,\h);
	\draw[line join=round] (\vx,\vy-0.4) node {{\small$v_1$}};
	\draw[line join=round] (\wx-0.3,\wy+0.3) node {\small$v_2'+v_2''$};
	\node [fill =white, inner sep=0pt, line join=round=2pt] at (\ux-0.7,\uy-0.1,\h) {\small$v_2'$};
	\draw[line join=round] (\ux+\vx+0.8,\vy+\uy,\h) node {\small$v_1+v_2'$};
	\draw[line join=round] (\wx+\vx+0.5,\vy+\wy+0.5) node {\small$v_1+v_2'+v_2''$};
\end{tikzpicture}
    \end{minipage}
    \begin{minipage}{0.4\textwidth}

\centering
\tikzmath{\z=5;}
\tikzmath{\x=2;}
\tikzmath{\y=6;}
\tikzmath{\p=1;}
\begin{tikzpicture}[scale=0.5]
	\draw[line join=round, dashed]  (0,\y,0) -- (0,0,0) -- (\z,\y,0) -- (\z,0,0);
	\draw[] (0,0,0) -- (\z,0,0) -- (\z,\y,0) --(0,\y,0);
	\draw[line join=round,] (0,0,0) -- (0,\x,\p) -- (0,\y,0)   -- (\z,\y,0) -- (\z,\x,\p) -- (\z,0,0) -- cycle;
	\draw[line join=round] (0,0,0) -- (\z,\x,\p) -- (0,\x,\p) -- (\z,\y,0);
	\draw[red, thick] (0,0,0) -- (\z,0,0);
\end{tikzpicture}
    \end{minipage}
	\caption{
		On the left, a ($3$-dimensional) representation of a lift in $\R^3$ of the $3$-chain $Q$ of \cref{lemma-splitting-vector}. The boundary of this lift is made up of: the two forward parallelograms, representing lifts of $\hprism[2][2]{v_1,v_2'}$ and $\hprism[2][2]{v_1,v_2''}$, the backward parallelogram, representing a lift of $\hprism[2][2]{v_1,v_2'+v_2''}$, and the two triangles on the left and on the right, given by $R(b_1^+)$ and $R(b_1^-)$, which simplify when projecting on $\T$. The red segment is the support of $\hprism[2][2]{v_1,\mathbf 0}$.\\
       On the right, it is represented the same picture in case of rectangle-like cycles (\cref{step-splitting-rect}).
	}\label{fig-slide-paral}
\end{figure}
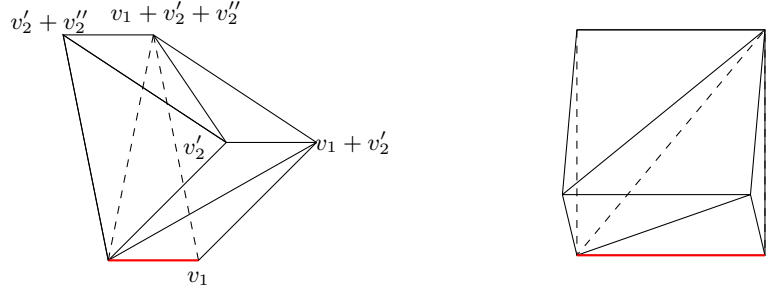\end{center}

The following is a direct application of the previous lemma in the particular case in which parallelogram-like cycles are rectangle-like.
\begin{lemma}\label{step-splitting-rect}
        Let $n\in\N$. There exists a constant $C_n$ such that  for any $a_1,\ldots,a_n, a_i',a_i''\in \Z$ with $a_i=a_i'+a_i''$ it holds that
    \[
        \fnorm{
            \hrect{a_1,\ldots,a_n}-\hrect{a_1,\ldots,a_i',\ldots,a_n}-\hrect{a_1,\ldots,a_i'',\ldots,a_n}
        }\leq C_n.
    \]
\end{lemma}

\subsection{Slim parallelogram-like cycles}\label{sec-slim}
In this subsection we deal with parallelogram-like cycles $\hprism{v_1,\ldots,v_k}$ whose generating vectors $v_1,\ldots,v_k\in \Z^n$ are linearly dependent. Another way to state this condition is by requiring that there exists a map $\iota\colon \T[k-1]\to \T[n]$ for which $\hprism{v_1,\ldots,v_k}$ is in the image of the map $\iota_*\colon \zcycles k {\T[k-1]}\to\zcycles{k}{\T[n]}$. 
As $\zhomol{k}{\T[k-1]}$ is trivial, these kinds of cycles are boundaries. 
In this subsection we control the filling norm of these parallelogram-like cycles in terms of their sizes.

We start with $2$-dimensional cycles in $S^1$.
\begin{lemma}\label{sistemareS1}
    There exists a constant $C_2>0$ such that for any nonzero integers $a,l$ it holds that
    \[\fnorm {\hprism[1][2]{a,l}}\le C_2\log_2(al)+C_2.\]
\end{lemma}

The following claims, together with \cref{lemma-splitting-vector} and \cref{lemma-zero-vector}, provide paths with controlled length inside the graph of cycles from the parallelogram-like cycle $\hprism[1][2]{a,l}$ to the zero chain.
\begin{claim}\label{claim-dehn}
    There exists a constant $C_2$ such that for every $x,y\in \Z$ and $k\in\{0,1,2,3\}$ it holds that $\fnorm{\hprism[1][2] {x,y} -\hprism[1][2] {x,y-kx}}\leq C_2$.
\end{claim}
\begin{proof}
    Set $c_2=\hrect[2]{1,1}$, and consider the Dehn twist $\phi\colon \T\to \T$  on the torus $\T$ along the curve $\gamma\colon S^1\to \T[2]$ given by $\gamma(t)=(t,0)$ (where we are using the identification $S^1= \R/\Z$ and $\T=\R^2/\Z^2$). As $c_2$ and $(\phi^{-k})_*(c_2)$ are fundamental cycles for $\T$, they are cobordant.
    We set
    \[C_2=\max_{k=0,1,2,3}\big\{\fnorm{c_2-(\phi^{-k})_*(c_2)}\big\}.\] 

    If $p_{x,y}\colon \T \to S^1$  denotes the map defined by $p((s,t))=xs+yt$, then
    \[
        \fnorm{\hprism[1][2] {x,y} -\hprism[1][2] {x,y-kx}}
        =\fnorm{(p_{x,y})_*(c_2-(\phi^{-k})_*(c_2))}
        \leq C_2
    \]
    holds for any $k=0,1,2,3$.
\end{proof}

\begin{claim}\label{claim-double-and-halve}
    There exists a constant $C_2$ such that for every $x\in \Z$ and $y\in2\Z$ it holds that $\fnorm{\hprism[1][2]{x,y} -\hprism[1][2]{2x,y/2}}\leq C_2$.
\end{claim}

This claim was already proved by Frigerio and the first author \cite[proof of Step 2]{bertolotti2024integral}.

Now we have all the ingredients to prove \cref{sistemareS1}.
\begin{proof}[Proof of \cref{sistemareS1}]
    Up to applying \cref{lemma-rearranging-vectors}, we can assume that  $a,l$ are two positive integers.
    We begin by showing that $\norm {\hprism[1][2] {1,l}}\leq  C_2\log_2 l+C_1+2C_2$.

    Set $x_0=1,y_0=l$. For $i\geq 0$, if $y_i \geq x_i$ we set 
    \[(x_{i+1},y_{i+1})= \left(2x_i, \frac{y_i-k_ix_i}2\right),\] 
    where $k_i\in \{0,1,2,3\}$ and $k_i\equiv y_i/2^i$ $\mathrm{mod}\ 4$.

    At each step $i$ we have $x_i=2^i$ and $y_i\leq l/2^{i}$, so that there exists an integer $M\leq 1+\frac{\log_2l}2$ such that $y_M< x_M$.

    Moreover, at the $i$-th step of the iteration, $y_i$ is an integer divisible by $2^i$ (so that the sequence is well-defined):
    indeed, this is true for $i=0$. Suppose by induction that for $1 \leq i\leq M$ the element $y_{i-1}$ is an integer divisible by $2^{i-1}$; by definition of $k_{i-1}$, there exists an integer $h$ such that $y_{i-1}=(k_{i-1}+4h)2^{i-1}$, so that 
    \[y_i=\frac{y_{i-1}-k_{i-1}x_{i-1}}{2}= 2^{i}h.\]
    In particular, as $0\leq y_M< x_M=2^M$ and $y_M$ is divisible by $2^M$, we have $y_M=0$.

    Now, by \cref{claim-dehn} and \cref{claim-double-and-halve} above we have that for every $i< M$ it holds that
    \begin{multline*}
        \fnorm{\hprism[1][2] {x_i,y_i} -\hprism[1][2]{x_{i+1},y_{i+1}}}\leq\\
        \leq \fnorm{\hprism[1][2] {x_i,y_i} -\hprism[1][2] {x_i,y_i-k_ix_i}}
        +\norm{\hprism[1][2] {x_i,y_i-k_ix_i}- \hprism[1][2]{x_{i+1},y_{i+1}}}
        \leq 2C_2,
    \end{multline*}
    and thus, by \cref{lemma-zero-vector}, we have that
    \begin{align*}
        \fnorm{\hprism[1][2] {1,l}}\leq   & \left(\sum_{i=0}^{M-1} \fnorm{\hprism[1][2] {x_i,y_i}-\hprism[1][2] {x_{i+1},y_{i+1}}}\right) + \fnorm{\hprism[1][2] {x_M,0}}\\
                                    \leq  & 2MC_2 + C_1 
                                    \leq    C_2\log_2 l+C_1+2C_2.
    \end{align*}

    Let us now consider the general case, with $a>1$ and $l\geq 1$. We are going to prove that there is some positive integer $L\le 2al$ for which $\fnorm{\hprism[1][2] {a,l} - \hprism[1][2] {1,L}}$ can be controlled (up to constants) by $\log_2(al)$. Then, since 
    \[\fnorm{\hprism[1][2]{a,l}}\le \fnorm{\hprism[1][2] {a,l} - \hprism[1][2] {1,L}}+\fnorm{\hprism[1][2] {1,L}},\]
    we can use the first part of the proof on $\hprism[1][2]{1,L}$ to conclude. 
    
    The idea is similar to the previous part of the proof: 
    we consider a recursive sequence of parallelogram-like chains $s_i=\hprism[1][2] {a_i,2^il_i}$, where $a_i$ is the integral part $\lfloor a_{i-1}/2\rfloor$ of $a_{i-1}/2$; if $a_{i-1}$ is odd we also get a remainder $r_i=\hprism[1][2]{1,2^il_i}$, intuitively obtained by cutting a ``slim'' parallelogram from $s_{i-1}$. At the end of this recursion, we get a chain obtained by the sum of the remainders $r_i$ with a chain $s_N$; all of them are parallelogram-like chains of type $\hprism[1][2] {1,x}$. Thanks to \cref{lemma-splitting-vector}, the sum of these parallelogram-like cycles is close (with respect to the filling norm) to a single parallelogram-like cycle $\hprism[1][2] {1,L}$.

    Consider the following sequence of numbers
    \[
        \begin{cases}
            a_0=a,\\
            a_i=\frac{a_{i-1}}{2} &\text{ if } a_{i-1} \text{ is even},\\
            a_i=\frac{a_{i-1}-1}{2} &\text{ if } a_{i-1} \text{ is odd.}
        \end{cases}
    \]
    Let $N$ be the integer for which $a_N=1$. Notice that $a_i\le a/2^{i}$ for all $i$, and thus $N\le \log_2(a)$.
    Let $I_e$ and $I_o$ be the sets of indices $i$ (with $0\leq i< N$) for which $a_i$ is even or odd respectively.

    Let us set
    \begin{align*}
        s_i     &=\hprism[1][2] {a_i,2^il}                &\text{for } 0\le i\le N,\\
        r_i     &=\hprism[1][2]{1,2^il}    &\text{for } i\in I_o. 
    \end{align*}
    Note that by \cref{claim-double-and-halve} for $i\in I_e$ it holds that
    \[
        \fnorm{s_i-s_{i+1}}  =
        \fnorm{\hprism[1][2]{a_i,2^il}-\hprism[1][2] {\frac{a_{i}}2,2^{i+1}l}}
        \leq C_2,
    \]
    while for $i\in I_o$, by \cref{lemma-splitting-vector} and \cref{claim-double-and-halve}, we have
    \begin{align*}
        \fnorm{s_i-r_i-s_{i+1}} 
            \leq    & + \fnorm{\hprism[1][2]{a_{i},2^{i}l} -\hprism[1][2]{1,2^{i}l} - \hprism[1][2] {a_{i}-1,2^{i}l}} \\
                    & + \fnorm{\hprism[1][2]{{a_{i}-1},2^il} - \hprism[1][2]{\frac{a_{i}-1}2,2^{i+1}l}} \\
            \leq    &  2C_2.
    \end{align*}
    If we set
    \[L\coloneqq l\left(2^N+\sum_{i\in I_o}2^i\right)<l\cdot2^{N+1}\leq 2la,\]
    then by inductively applying \cref{lemma-splitting-vector}, one also gets 
    \begin{align*}
        \norm{\hprism[1][2] {1,L} -s_N -\sum_{i\in I_o}r_i}[\fil,\Z]\leq C_2 \cdot|I_o|.
    \end{align*}

    Thus, by noticing that
    \begin{multline*}
        \hprism[1][2] {a,l} -\hprism[1][2] {1,L}   
        =  \sum_{i\in I_e}(s_i-s_{i+1}) + \sum_{i\in I_o}(s_i-s_{i+1}-r_i)-\left(\hprism[1][2] {1,L} -s_N-\sum_{i\in I_o} r_i \right),
    \end{multline*}
    we get
    \begin{align*}
        \fnorm{\hprism[1][2] {a,l}}
        \leq & \fnorm{\hprism[1][2] {a,l}-\hprism[1][2] {1,L}} +\fnorm{\hprism[1][2] {1,L}} \\
        \leq & + \sum_{i\in I_e}\fnorm{s_i-s_{i+1}} + \sum_{i\in I_o}\fnorm{s_i-s_{i+1}-r_i} \\
             & + \fnorm{\hprism[1][2] {1,L} -s_N-\sum_{i\in I_o} r_i}+ \fnorm{\hprism[1][2] {1,L}} \\
        \leq & C_2\cdot |I_e| + 2C_2\cdot |I_o| + C_2\cdot |I_o| +  C_2\log_2 L+C_1+2C_2\\
        \leq & 4C_2N+C_2\log(2al)+C_1+2C_2\\
        \leq & 4C_2\log_2 a+C_2\log(2al)+C_1+2C_2,
    \end{align*}
    as desired.
    
\end{proof}

We now prove the same statement in any dimension.
\begin{lemma}\label{lemma-slim-parallelogram-hdim}
    There exists a constant $C_k$ such that for any integers $k> n\geq 1$ and any $v_1,\ldots, v_k\in \Z^n$ it holds that
    \[\fnorm {\hprism[n]{v_1,\ldots,v_k}}\leq C_k\cdot\log_2 (\max \{\norm{v_1}[\infty],\ldots,\norm{v_k}[\infty]\})+C_k. \]
\end{lemma}
\begin{proof}
    We first prove the statement when $k=n+1$ by induction on $n$: \cref{sistemareS1} already gives the statement on $\T[1]=S^1$ since $\log_2(al)\le2\log_2(\max\{|a|,|l|\})$, so let us fix $n\geq 2$ and assume that there exists a constant $V_{n}$ as in the statement; the goal is to prove the result in dimension $n+1$. 

    \begin{claim}\label{claim-sustet-dependent-vector}
        If $v_1, \ldots,\widehat v_i,\ldots, v_{n+1}$ are linearly dependent for some $i=1,\ldots,n+1$, then 
        \[
            \fnorm{\hprism[n][n+1]{v_1,\ldots, v_{n+1}}}
            \leq (n+2)V_n\cdot\log_2(\max \{\norm{v_1}[\infty],\ldots,\norm{v_{n+1}}[\infty]\})+(n+2)V_n+C_{n+1}.
        \]
    \end{claim}
    Indeed, in this case, the par\-al\-lel\-o\-gram-like cycle $\hprism[n][n]{v_1,\ldots,\widehat{v}_i,\ldots,v_{n+1}}$ is contained in some torus of dimension $n-1$ inside $\T[n]$. By inductive hypothesis, \cref{lemma-prism-op-on-boundary,lemma-rearranging-vectors}, we have
    \begin{align*}
        \fnorm{\hprism[n][n+1]{v_1,\ldots, v_{n+1}}}
        &\leq\fnorm{\hprism[n][n+1]{v_1,\ldots,\widehat v_i,\ldots, v_{n+1},v_i}}+C_{n+1}\\
        &=\fnorm{\po{v_{i}}(\hprism[n][n]{v_1,\ldots,\widehat v_i,\ldots,v_{n+1}})}+C_{n+1}\\
        &\leq(n+2)\fnorm{\hprism[n][n]{v_1,\ldots,\widehat v_i,\ldots,v_{n+1}}}+C_{n+1}\\
        &\leq (n+2)V_n\cdot\log_2(\max \{\norm{v_1}[\infty],\ldots,\norm{v_{n+1}}[\infty]\})+(n+2)V_n+C_{n+1},
    \end{align*}
    proving the claim.

    Let us now consider $v_1,\ldots, v_{n+1}\in \Z^{n}$, with $v_1, \ldots, v_{n}$ linearly independent, and let $v_{n+1}=v_{n+1}^\parallel+v_{n+1}^\perp$, where $v_{n+1}^\parallel$ is parallel to $v_n$, while $v_{n+1}^\perp$ is contained in the subspace generated by $v_1,\ldots,v_{n-1}$. Then, as the sets $\{v_1,\ldots,v_{n-1},v_{n+1}^\perp\}$ and $\{v_2,\ldots,v_{n},v_{n+1}^\parallel\}$ consist of linearly dependent vectors, one gets from \cref{lemma-splitting-vector}, \cref{claim-sustet-dependent-vector}, and the triangular inequality that
    \begin{multline*}
        \fnorm{\hprism[n][n+1]{v_1,\ldots, v_{n+1}}}\leq\\
        \begin{aligned}
            \leq& + \fnorm{\hprism[n][n+1]{v_1,\ldots,v_{n+1}} - \hprism[n][n+1]{v_1,\ldots,v_n,v_{n+1}^\perp}-\hprism[n][n+1]{v_1,\ldots,v_{n},v_{n+1}^\parallel}} \\
                & +\fnorm{\hprism[n][n+1]{v_1,\ldots,v_n,v_{n+1}^\perp}}+\fnorm{\hprism[n][n+1]{v_1,\ldots,v_{n},v_{n+1}^\parallel}}\\
            \leq& C_{n+1}+ 2(n+2)V_n\cdot\log_2(\max \{\norm{v_1}[\infty],\ldots,\norm{v_{n+1}}[\infty]\})+2(n+2)V_n+2C_{n+1},
        \end{aligned}
    \end{multline*}
    as desired.

    Consider now $k\geq n+1$,  and by induction we suppose that  there exists a constant $V_{k}$ as in the statement. Then, for any $v_1,\ldots,v_{k+1}\in\Z^n$ one has
    \begin{align*}
        \fnorm{\hprism[n]{v_1,\ldots,v_{k+1}}}
        =&\fnorm{\po{v_{k+1}}(\hprism[n][k]{v_1,\ldots,v_{k}})}\\
        \leq& (k+2)\fnorm{\hprism[n]{v_1,\ldots,v_{k}}}\\
        \leq& (k+2)V_{k}\log_2(\max \{\norm{v_1}[\infty],\ldots,\norm{v_{k}}[\infty]\})+(k+2)V_{k},
   \end{align*}
   as desired.
\end{proof}

\subsection{Rectangle-like cycles}\label{sec-rect}
 In this subsection we explore properties that hold specifically for rect\-an\-gle-like cycles.

The next lemma shows that any $n$-dimensional parallelogram-like cycle in $\T[n]$ is close to a sum of rectangle-like cycles. Recall that the sizes of a rectangle-like cycle $\hrect{a_1,\ldots,a_n}$ are the integers $a_1,\ldots,a_n\in\Z$. 
\begin{lemma}\label{lemma-hprism-sum-rect}
    For any positive integer $n\in\N$ there exists a constant $C_n$ such that the following holds: for any $v_1,\ldots,v_n\in \Z^n$ there are $n!$ rectangle-like cycles $r_1,\ldots,r_{n!}$ whose sizes are bounded from above by $\max\{\norm{v_1}_{\infty},\ldots, \norm{v_n}_{\infty}\}$ and such that
    \[
        \fnorm{
            \hprism[n][n]{v_1,\ldots,v_n}-\sum_{i=1}^{n!}\epsilon_ir_i
        }\leq C_n\cdot\log_2 (\max\{\norm{v_1}_{\infty},\ldots, \norm{v_n}_{\infty}\})+C_n,
    \]
    for some $\epsilon_i\in\{\pm 1\}$.
\end{lemma}
\begin{proof}
    The argument is by induction on the dimension: the statement is clearly true in dimension $1$.
    We fix $n>1$ and suppose  that there is a constant $V_m$ as in the statement in any dimension $1\leq m<n$. Let $e_1,\ldots,e_n$ be the canonical basis of $\R^n$, and let $v_1,\ldots,v_n\in \Z^n$. 
    For any $i=1,\ldots, n$, let $v_i^\perp$ be the projection of $v_i$ on the plane orthogonal to $e_n$ and $v_i^{\parallel}=v_i-v_i^\perp$.

    Let $\mathcal{P}_n$ be the power set of $\{1,\ldots,n\}$ and let $\mathcal{U}_n\subset \mathcal P_n$ be the  subset consisting only of one-element subsets.
    For any $S\in \mathcal{P}_n$ we denote by $\mathbf{v}^S$ the $n$-tuple $(v_1^{\theta_1},\ldots,v_n^{\theta_n})$, where $\theta_i=\parallel$ if $i\in S$, otherwise $\theta_i=\perp$.

    For any $S_i\coloneqq\{i\}\in\mathcal {U}_n$, the $n$-tuple $\mathbf v^{S_i}$ has $n-1$ vectors $v_1^\perp,\ldots,v_{i-1}^\perp,v_{i+1}^\perp,\ldots, v_n^\perp$ contained in the $(n-1)$-dimensional subspace of $\R^n$ orthogonal to $e_n$. 
    Thus, the parallelogram-like cycle $\hprism[n][n-1]{v_1^\perp,\ldots,v_{i-1}^\perp,v_{i+1}^\perp,\ldots, v_n^\perp}$ can be seen as a parallelogram-like cycle contained in the embedded torus $\T[n-1]$ inside $\T[n]$ corresponding to the subspace of $\R^{n}$ orthogonal to $e_{n}$. 
    By applying the inductive hypothesis, we get a chain $c_i'\in \zcycles {n-1}{\T[n-1]}\subset\zcycles {n-1}{\T[n]}$
    that is a sum (or difference) of at most $(n-1)!$ rectangle-like cycles whose sizes are bounded from above by 
    \[\max\left\{\norm{v_1^\perp}_{\infty},\ldots,\norm{v_{i-1}^\perp}_{\infty},\norm{v_{i+1}^\perp}_{\infty},\ldots, \norm{v_n^\perp}_{\infty}\right\}\le\max\left\{\norm{v_1}_{\infty},\ldots, \norm{v_n}_{\infty}\right\}\]
    and such that
    \[
        \fnorm{\hprism[n][n-1]{v_1^\perp,\ldots,v_{i-1}^\perp,v_{i+1}^\perp,\ldots, v_n^\perp}-c_i'}\leq V_{n-1} \log_2 (\max\{\norm{v_1}_{\infty},\ldots, \norm{v_n}_{\infty}\})+V_{n-1}.
    \]

    Let us set $ c_i \coloneqq-\po {v_i^{\parallel}}(c_i')$ for $i=1,\ldots,n-1$ and $c_n\coloneqq\po {v_n^{\parallel}}(c_n')$.
    Being $c_i'$ a sum (or difference) of rectangle-like cycles of dimension $n-1$ whose canonical lifts are contained in the subspace orthogonal to $e_n$, the cycle $ c_i $ is still a sum (or difference) of rectangle-like cycles for any $i=1,\ldots,n$. 
    By \cref{lemma-prism-op-on-boundary} and \cref{lemma-rearranging-vectors}, for $i\neq n$ we have that
    \begin{align*}
        \fnorm{\hprism[n][n]{\mathbf{v}^{S_i}}-c_i}
        \leq &+\fnorm{\hprism[n][n]{\mathbf{v}^{S_i}} + \hprism[n][n]{v_1^\perp,\ldots, v_{i-1}^\perp,v_{i+1}^\perp,\ldots, v_n^\perp,v_i^\parallel}}\\
             &+\fnorm{-\hprism[n][n]{v_1^\perp,\ldots, v_{i-1}^\perp,v_{i+1}^\perp,\ldots, v_n^\perp,v_i^\parallel}-c_i}\\
        \leq & C_{n} + \fnorm{-\po{v_i^{\parallel}}\big(\hprism[n][n]{v_1^\perp,\ldots, v_{i-1}^\perp,v_{i+1}^\perp,\ldots, v_n^\perp}-c_i'\big)}\\
        \leq & C_n +(n+1)\fnorm{\hprism[n][n]{v_1^\perp,\ldots, v_{i-1}^\perp,v_{i+1}^\perp,\ldots, v_n^\perp}-c_i'}\\
        \leq &C_n +(n+1) V_{n-1} \log_2(\max\{\norm{v_1}_{\infty},\ldots, \norm{v_n}_{\infty}\})+(n+1) V_{n-1},
    \end{align*}
    while for $i=n$ we have (analogously as before) that
        \begin{align*}
        \fnorm{\hprism[n][n]{\mathbf{v}^{S_n}}-c_n}
        \leq &\fnorm{\hprism[n][n]{v_1^\perp,\ldots, v_{n-1}^\perp,v_n^\parallel}-c_n}\\
        \leq &(n+1) V_{n-1} \log_2(\max\{\norm{v_1}_{\infty},\ldots, \norm{v_n}_{\infty}\})+(n+1) V_{n-1}.
    \end{align*}
    Thus, by setting $c\coloneqq \sum_{i=1}^{n} c_i$ and by applying iteratively the triangular inequality, one gets 
\begin{multline*}
	\fnorm{\hprism[n][n]{v_1,\ldots,v_n}-c}\leq \\
	\begin{aligned}
		\leq&\fnorm{\hprism[n][n]{v_1,\ldots,v_n}-\sum_{i=1}^n\hprism[n][n]{\mathbf v^{S_i}}}+\sum_{i=1}^n\fnorm{\hprism[n][n]{\mathbf v^{S_i}}- c_i}\\
		\leq&+\fnorm{\hprism[n][n]{v_1,\ldots,v_n}-\sum_{S\in \mathcal P_n}\hprism[n][n]{\mathbf v^S}} + \sum_{S\in \mathcal P_n\setminus \mathcal U_n}\fnorm{\hprism[n][n]{\mathbf v^S}}\\
		&+nC_n +n(n+1) V_{n-1} \log_2(\max\{\norm{v_1}_{\infty},\ldots, \norm{v_n}_{\infty}\})+n(n+1) V_{n-1}.
	\end{aligned}
\end{multline*}

	We deal with  $\fnorm{\hprism[n][n]{v_1,\ldots,v_n}-\sum_{S\in \mathcal P_n}\hprism[n][n]{\mathbf v^S}}$  by applying iteratively \cref{lemma-splitting-vector} to $\hprism[n][n]{v_1,\ldots,v_n}=\hprism[n][n]{v_1^{\parallel}+v_1^\perp,\ldots,v_n^{\parallel}+v_n^\perp}$ for each generating vector.
    
    For any set $S\in \mathcal P_n\setminus \mathcal U_n$, the $n$-tuple $\mathbf v^S$ consists of linearly dependent vectors whose $\ell^\infty$-norm is bounded from above by $\max\{\norm{v_1}_{\infty},\ldots, \norm{v_n}_{\infty}\}$ and thus, by applying \cref{lemma-slim-parallelogram-hdim}, we also get an upper bound for $\fnorm{\hprism[n][n]{\mathbf v^S}}$ for each $S\in P_n\setminus \mathcal U_n$. 
    
    Putting everything together we get the statement.
\end{proof}

Rectangle-like cycles are useful because they are close to rectangle-like cycles with all sizes but one equal to $1$.

\begin{lemma}\label{lemma-rect-with-size1-hdim}
    There exists a constant $C_n$ such that for any integers $a_1,\ldots, a_n\in \Z$, not all zero, it holds that
    \[ \fnorm{\hrect{a_1,\ldots,a_n}-\hrect{\prod_{i=1}^{n}a_i,1,\ldots,1}}\leq C_n\log_2(\max\{|a_1|,\ldots,|a_n|\})+C_n.\]
\end{lemma}
This lemma is proved via an inductive argument on the dimension $n$. In order to prove the base case $n=2$ we need the following claim, which is a consequence of \cref{lemma-splitting-vector}, and allows us to ``slide'' a parallelogram-like cycle along one of its sides (see \cref{fig-rect-unit-height}).

\begin{claim}\label{step-slide-paral-along-a-side}
    There exists a constant $C$ such that for any vectors $v,w\in \Z^2$ and any integers $n,d\in \Z\setminus\{0\}$ with $\frac nd v\in \Z^2$, it holds that 
    \[\fnorm{\hprism[2][2]{v,w+\frac nd v} -\hprism[2][2] {v,w} }\le  C \log_2|dn|+C\]
    and
    \[\fnorm{\hprism[2][2] {w,v} - \hprism[2][2] {w+\frac nd v,v}}\le C\log_2|dn|+C.\]
\end{claim}
\begin{proof}
    Indeed, by the triangular inequality we have 
    \begin{multline*}
        \fnorm{\hprism[2][2]{v,w+\frac nd v} -\hprism[2][2] {v,w} }\leq \\
        \leq \fnorm{\hprism[2][2]{v,w+\frac nd v} - \hprism[2][2] {v,w}-\hprism[2][2] {v,\frac nd v}}      
        +\fnorm{\hprism[2][2] {v,\frac nd v}},
    \end{multline*}
    where 
    \[\fnorm{\hprism[2][2]{v,w+\frac nd v} - \hprism[2][2] {v,w}-\hprism[2][2] {v,\frac nd v}}\leq C_2 \] 
    due to \cref{lemma-splitting-vector}.
    Moreover, if we define $\iota_{v}\colon S^1\to \T$ by  $\iota_v(t)=\frac t d v$ (where we are using the identifications $S^1=\R/\Z$ and $\T=\R^2/\Z^2$), then $(\iota_v)_*(\hprism[1][2] {d,n})=\hprism[2][2] {v,\frac nd v}$.
    By \cref{sistemareS1} we get
        \[\fnorm{\hprism[2][2] {v,\frac nd v}}\leq \fnorm {\hprism[1][2] {d,n}} \leq C_2 \log_2|{dn}|+C_2 .\]
\end{proof}
We can now proceed with the proof of \cref{lemma-rect-with-size1-hdim}.

\begin{proof}[Proof of \cref{lemma-rect-with-size1-hdim}]
    The statement is trivial for $n=1$. Moreover, notice that if $a_i=0$ for some $i$, then the statement is a consequence of \cref{lemma-rearranging-vectors} and \cref{lemma-zero-vector}.
    Let us then suppose $n\geq 2$ and $a_1,\ldots, a_n\in\Z\setminus\{0\}$. 
    
    As said before, the proof is an induction on $n$ with base case $n=2$.
    Let us start by proving the result for $n=2$ in the case $a_1\neq 1$. We apply \cref{step-slide-paral-along-a-side} four times as follows (see \cref{fig-rect-unit-height}):
    \begin{enumerate}
        \item with the notation of \cref{step-slide-paral-along-a-side}, consider $v=a_1e_1$, $w=a_2e_2$, $n=a_2$ and $d=a_1$; then, recalling that $\hrect[2]{a_1,a_2}= \hprism[2][2]{a_1e_1,a_2e_2}$, we have that 
        \begin{align*}
            \fnorm{\hrect[2] {a_1,a_2}-\hprism[2][2] {a_1e_1,a_2e_1+a_2e_2}}\le  C \log_2 |a_1a_2|+C;
        \end{align*}
        
        \item consider $v= {a_1e_1}$, $w={a_2e_1+a_2e_2}$, $n=-1$ and $d=a_2$; then
        \[
        \fnorm{\hprism[2][2]{a_1e_1,a_2e_1+a_2e_2}-\hprism[2][2]{(a_1-1)e_1-e_2,a_2e_1+a_2e_2}}\le  C \log_2 |a_2|+C;
        \]	
        
        \item consider $v= {(a_1-1)e_1-e_2}$, $w={a_2e_1+a_2e_2}$, $n=a_2$ and $d=1$; then 
        \[
        \fnorm{\hprism[2][2]{(a_1-1)e_1-e_2,a_2e_1+a_2e_2}-\hprism[2][2]{(a_1-1)e_1-e_2,a_1a_2e_1}}\le C \log_2 |a_2|+C;
        \]
        
        \item\label{item-to-skip-in-special-case} consider $v= {(a_1-1)e_1-e_2}$, $w=a_1a_2e_1$, $n=-a_1+1$ and $d=a_1a_2$; then
            \[
        \fnorm{\hprism[2][2]{(a_1-1)e_1-e_2,a_1a_2e_1}-\hprism[2][2]{-e_2,a_1a_2e_1}}\le C \log_2 |(1-a_1)a_1a_2|+C.
        \]
    \end{enumerate}
    Moreover, by applying twice \cref{lemma-rearranging-vectors} we have that
    \[\fnorm{\hprism[2][2]{-e_2,a_1a_2e_1}+\hprism[2][2]{a_1a_2e_1,e_2}}\leq 2C_2.\]

    Putting everything together, by the triangle inequality we get 
    \begin{multline*}
        \fnorm{\hrect[2]{a_1,a_2}-\hrect[2]{a_1a_2,1}}
        \leq \\
        {\begin{aligned}
            \leq &+\fnorm{\hrect[2] {a_1,a_2}-\hprism[2][2] {a_1e_1,a_2e_1+a_2e_2}}\\
                & +\fnorm{\hprism[2][2] {a_1e_1,a_2e_1+a_2e_2}-\hprism[2][2] {(a_1-1)e_1-e_2,a_2e_1+a_2e_2}}\\
                &+\fnorm{\hprism[2][2] {(a_1-1)e_1-e_2,a_2(e_1+e_2)}-\hprism[2][2] {a_1e_1-e_1-e_2,a_1a_2e_1}}\\
                & +\fnorm{\hprism[2][2] {(a_1-1)e_1-e_2,a_1a_2e_1}-\hprism[2][2] {-e_2,a_1a_2e_1}}\\
                & +\fnorm{\hprism[2][2] {-e_2,a_1a_2e_1}-\hrect[2]{a_1a_2,1}}\\
            \leq &\ 4C \log_2 |(1-a_1)a_1a_2|+4C+2C_2\\
            \leq &\ 4C \log_2 (2|a_1a_2|^2)+4C+2C_2,
        \end{aligned}}
    \end{multline*}
    proving the statement for $a_1\neq 1$ and $n=2$.

    On the other hand, if $a_1=1$ and $n=2$, then we can apply the same proof without the fourth application of \cref{step-slide-paral-along-a-side} (that is \cref{item-to-skip-in-special-case}).

    Let us now consider $n\geq 3$, suppose there is a constant $V_m$ in any dimension $m<n$ as in the statement, and consider $a_1,\ldots, a_n\in \Z\setminus\{0\}$. Then, it holds that
    \begin{multline*}
        \fnorm{\hrect{a_1,\ldots,a_n}-\hrect{\prod_{i=1}^{n}a_i,1,\ldots,1}}\leq\\
        {\begin{aligned}
            \leq&+ \fnorm{\hrect{a_1,\ldots,a_n}-\hrect{a_1a_n,a_2,\ldots, a_{n-1},1}}\\
                &+\fnorm{\hrect{a_1a_n,a_2,\ldots,a_{n-1},1}-\hrect{\prod_{i=1}^{n}a_i,1,\ldots,1}}.
        \end{aligned}}
    \end{multline*}
    By the $2$-dimensional case (where we regard $\hprism[n][2]{a_1e_1,a_ne_n}$ and $\hprism[n][2]{a_1a_ne_1,e_n}$ as rectangle-like cycles contained in a $2$-dimensional torus embedded in $\T[n]$), \cref{lemma-rearranging-vectors} and an iterative application of the prism operator (\cref{lemma-paral-are-cycle}) we have that
    \begin{multline*}
        \fnorm{\hrect{a_1,\ldots,a_n}-\hrect{a_1a_n,a_2,\ldots, a_{n-1},1}}\leq\\
        {\begin{aligned}
            \leq & +\fnorm{\hrect{a_1,\ldots,a_n}+\hprism[n][n]{a_1e_1,a_ne_n,a_2e_2,\ldots, a_{n-1}e_{n-1}}}\\
                 & +\fnorm{\hprism[n][n]{a_1a_ne_1,e_n,a_2e_2,\ldots, a_{n-1}e_{n-1}}-\hprism[n][n]{a_1e_1,a_ne_n,a_2e_2,\ldots, a_{n-1}e_{n-1}}}\\
                 & +\fnorm{-\hprism[n][n]{a_1a_ne_1,e_n,a_2e_2,\ldots,a_{n-1}e_{n-1}}-\hrect{a_1a_n,a_2,\ldots, a_{n-1},1}}\\
            \leq & 2C_n+\fnorm{\po{a_{n-1}e_n}\big(\cdots \big(\po{a_2e_2}(\hprism[n][2]{a_1e_1,a_ne_n}-\hprism[n][2]{a_1a_ne_1,e_n})\big)\cdots\big)}\\
            \leq & 2C_n+(n+1)!\cdot V_2\log_2(\max\{|a_1|,|a_n|\})+(n+1)!\cdot V_2.
        \end{aligned}}
    \end{multline*}

    On the other hand, as above, the parallelogram-like cycles $\hprism[n][n-1]{a_1a_n e_1,a_2e_2,\ldots,a_{n-1}e_{n-1}}$ and $\hprism[n][n-1]{a_1\cdots a_ne_1,e_2,\ldots,e_{n-1}}$ can be regarded as rectangle-like cycles contained in a $(n-1)$-dimensional torus embedded in $\T[n]$. By applying the inductive hypothesis one gets
    \begin{multline*}
        \fnorm{\hrect{a_1a_n,a_2,\ldots,a_{n-1},1}-\hrect{a_1\cdots a_n,1,\ldots,1}}=\\
        \begin{aligned}
            =&\fnorm{\po {e_n} \left(\hprism[n][n-1]{a_1a_n e_1,a_2e_2,\ldots,a_{n-1}e_{n-1}}-\hprism[n][n-1]{a_1\cdots a_ne_1,e_2,\ldots,e_{n-1}}\right)}\\
            \leq& (n+1)V_{n-1}\log_2(\max\{|a_1a_n|,|a_2|,\ldots,|a_n|\})+(n+1)V_{n-1}.
        \end{aligned}
    \end{multline*}
    Putting the three inequalities above together, we get the desired bound.
\end{proof}

    \begin{center}\begin{figure}[!h]
 
\centering
\tikzmath{\x=3;}
\tikzmath{\y=4;}
\begin{tikzpicture}[scale=0.6]
	\draw[->,line join=round, thick, color =gray!20,] (0,0) -- (\x+\y+1,0);
	\draw[->,line join=round, thick, color = gray!20] (0,0) -- (0,\y+1);
	\draw[gray!50,dashed] (\x+\y,0)--(\x+\y,\y);
	\draw[gray!50,dashed] (\y,0)--(\y,\y);
	\draw[gray!50,] (0,\y)--(\x+\y,\y);
	
	
	\draw[line join=round,thick] (0,0) -- (\x,\y) -- (0,\y) -- (0,0) -- (\x,0) -- (\x,\y);
	
	\draw[line join=round,red] (0,0) -- (\x+\y,\y) -- (\y,\y) -- (0,0) -- (\x,0) -- (\x+\y,\y);
	
	\draw[] (\x,-0.3) node {\small $a_1$};
	\draw[] (\y,-0.3) node {\small $a_2$};
	\draw[] (\x+\y,-0.3) node {\small $a_1+a_2$};
	\draw[] (-0.3,\y) node {\small $a_2$};
	\begin{scope}[shift= {(\x+\y+3,0)}]
		\draw[->,line join=round, thick, color =gray!20,] (0,0) -- (\x+\y+1,0);
		\draw[->,line join=round, thick, color =gray!20,] (0,-1.5) -- (0,\y+1);
		\draw[gray!50,dashed] (\x+\y-1,0)--(\x+\y-1,\y-1);
		\draw[gray!50,dashed] (\x-1,0)--(\x-1,-1);
		\draw[gray!50,dashed] (0,\y)--(\y,\y);
		\draw[gray!50,dashed] (\y,0)--(\y,\y);
		\draw[gray!50,dashed] (0,\y-1)--(\y+\x-1,\y-1);
		\draw[line join=round,gray!50,dashed] (0,-1)--(\x-1,-1);
		\draw[line join=round,thick] (0,0) -- (\x+\y,\y) -- (\y,\y) -- (0,0) -- (\x,0) -- (\x+\y,\y);
		\draw[line join=round,red] (0,0) -- (\x+\y-1,\y-1) -- (\y,\y) -- (0,0) -- (\x-1,-1) -- (\x+\y-1,\y-1);
		\draw[line join=round] (\x-1.2,+0.3) node {\small $a_1-1$};
		\draw[line join=round] (\x+\y-1,-0.3) node {\small $a_1+a_2-1$};
		\draw[line join=round] (-0.3,\y) node {\small $a_2$};
		\draw[line join=round] (\y,-0.3) node {\small $a_2$};
		\draw[line join=round] (-0.8,\y-1) node {\small $a_2-1$};
		\draw[line join=round] (-0.3,-1) node {\small $-1$};
	\end{scope}
	\begin{scope}[shift= {(0,-7)}]
		\draw[->,line join=round, thick, color =gray!20,] (0,0) -- (\x*\y+\x,0);
		\draw[->,line join=round, thick, color =gray!20,] (0,-1.5) -- (0,\y+1);
		\draw[gray!50,dashed] (0,-1)--(\x-1,-1);
		\draw[gray!50,dashed] (\x*\y+\x-1,0) -- (\x*\y+\x-1,-1);
		\draw[gray!50,dashed] (\x-1,0)--(\x-1,-1);
		\draw[gray!50] (\y,\y) -- (\x*\y+\x-1,-1);
		\draw[line join=round,thick] (0,0) -- (\x+\y-1,\y-1) -- (\y,\y) -- (0,0) -- (\x-1,-1) -- (\x+\y-1,\y-1);
		\draw[line join=round,red] (0,0) -- (\x*\y+\x-1,-1) -- (\x*\y,0) -- (0,0) -- (\x-1,-1) -- (\x*\y+\x-1,-1);
		\draw[line join=round] (\x-1.2,+0.3) node {\small $a_1-1$};
		\draw[line join=round] (\x*\y+\x-0.7,+0.3) node {\small $a_1a_2+a_1-1$};
		\draw[line join=round] (\x*\y-0.3,+0.3) node {\small $a_1a_2$};
		\draw[line join=round] (-0.3,-1) node {\small $-1$};
	\end{scope}
	\begin{scope}[shift= {(0,-12)}]
		\draw[->,line join=round, thick, color =gray!20,] (0,0) -- (\x*\y+\x,0);
		\draw[->,line join=round, thick, color =gray!20,] (0,-1.5) -- (0,+1);
		\draw[line join=round,gray!50,dashed] (\x*\y+\x-1,0) -- (\x*\y+\x-1,-1);
		
		\draw[line join=round,thick] (0,0) -- (\x*\y+\x-1,-1) -- (\x*\y,0) -- (0,0) -- (\x-1,-1) -- (\x*\y+\x-1,-1);
		\draw[line join=round,red] (0,0) -- (\x*\y,-1) -- (\x*\y,0) -- (0,0) -- (0,-1) -- (\x*\y,-1);
		\draw[line join=round] (\x-1.2,+0.3) node {\small $a_1-1$};
		\draw[line join=round] (\x*\y+\x-.7,+0.3) node {\small $a_1a_2+a_1-1$};
		\draw[line join=round] (\x*\y-0.3,+0.3) node {\small $a_1a_2$};
		\draw[line join=round] (-0.3,-1) node {\small $-1$};
	\end{scope}
\end{tikzpicture}
        \caption{
            To prove \cref{lemma-rect-with-size1-hdim} for $n=2$, we apply the slidings of \cref{step-slide-paral-along-a-side} four times as shown in the picture, where at each step the black parallelogram is the starting one, and the red parallelogram is the one we obtain after the application of \cref{step-slide-paral-along-a-side}. \\
            Notice that in the case $a_1=1$, we already have a rectangle like cycle at the third application of \cref{step-slide-paral-along-a-side}.
        }\label{fig-rect-unit-height}
    \end{figure}
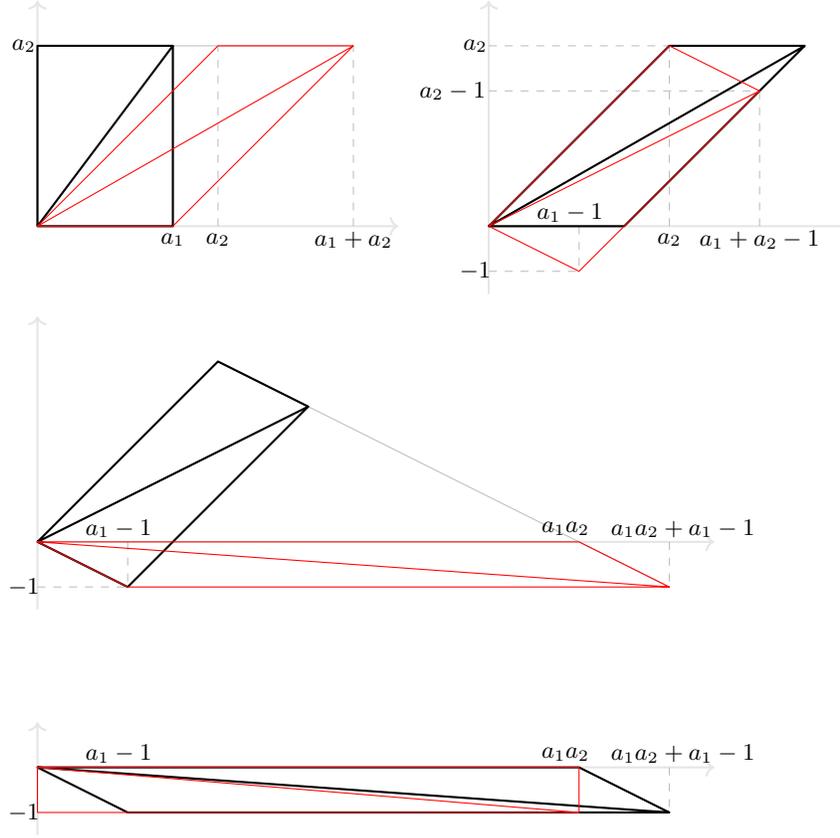\end{center}

\subsection{Proof of \cref{metteresuS1}}\label{sec-proof-metteresuS1}
Now that we have all the ingredients, we are ready to draw the path in the graph of cycles that goes from the parallelogram-like cycle $\hprism[n][n]{v_1,\ldots,v_n}$ to the rectangle-like cycle $\hrect {\det(A),1,\ldots,1}$, where $A$ is the matrix whose columns are given by $v_1,\ldots,v_n$.
This will prove \cref{metteresuS1}, which we recall in the following for the convenience of the reader.
\propupperbound*

The steps of this journey go as follows: we start from $\hprism[n][n]{v_1,\ldots,v_n}$ and go to a sum of rectangle-like cycles (\cref{lemma-hprism-sum-rect}); then, from each of these rectangle-like cycles we go to rectangle-like cycles of unit-height (\cref{lemma-rect-with-size1-hdim}), that we can sum to get a single rectangle-like cycle of unit-height (\cref{step-splitting-rect}).
The following proof makes the journey just described more precise.
\begin{proof}
    By \cref{lemma-hprism-sum-rect}, there exists a sum $\sum_{i=1}^{n!}\epsilon_ir_i'$, where $\epsilon_i\in\{\pm1\}$ and $r_i'$ is a rectangle-like cycle with sizes bounded from above by $\max\{\norm{v_1}_{\infty},\ldots, \norm{v_n}_{\infty}\}$ for any $i=1,\ldots, n!$, such that 
    \[
        \fnorm{
            \hprism[n][n]{v_1,\ldots,v_n}-\sum_{i=1}^{n!}\epsilon_i r_i'
        }\leq C_n\cdot\log_2 (\max\{\norm{v_1}_{\infty},\ldots, \norm{v_n}_{\infty}\})+C_n.
    \]

    By \cref{lemma-rect-with-size1-hdim}, for any $r_i'$ there exists a rectangle-like cycle 
    \[r_i\coloneqq\hrect{\ell_i,1,\ldots,1},\]
    for which 
    \[\fnorm{r_i'-r_i}\leq C_n\log_2(\max\{\norm{v_1}_{\infty},\ldots, \norm{v_n}_{\infty}\})+C_n.\]
    
    Moreover, by applying \cref{step-splitting-rect} iteratively to the sum $\sum_{i=1}^{n!} \epsilon_i r_i$, we obtain a rectangle-like cycle 
    \[R\coloneqq\hrect{L,1,\ldots,1},\]
    with $L=\sum_{i=1}^{n!} \epsilon_i\ell_i$, for which
    \[
        \fnorm{R-\sum_{i=1}^{n!} \epsilon_i r_i}\leq n!C_n.
    \]

    Thus, by the triangular inequality, we have
    \begin{align*}
        \fnorm{\hprism[n][n]{v_1,\ldots,v_n}-R}
        \leq&+\fnorm{\hprism[n][n]{v_1,\ldots,v_n}-\sum_{i=1}^{n!}\epsilon_i r_i'}\\
            &+\sum_{i=1}^{n!}\fnorm{\epsilon_i (r_i'-r_i)}+\fnorm{\sum_{i=1}^{n!} \epsilon_ir_i- R}\\
        \leq& (1+n!)C_n\log_2(\max\{\norm{v_1}_{\infty},\ldots, \norm{v_n}_{\infty}\})+(1+2n!)C_n.
    \end{align*}

    In particular, $\hprism[n][n]{v_1,\ldots,v_n}$ and $R$ are cobordant, so that if $A$ (resp.~$B$) denotes the matrix with columns $v_1,\ldots, v_n$ (resp.~$L e_1,e_2\ldots, e_n$), then by \cref{paral-homology} we get $\det(A)=\det (B)=L$. This concludes the proof.
\end{proof}

\addcontentsline{toc}{chapter}{Bibliography}

\bibliographystyle{alpha}
\bibliography{bibliografia} 
\nocite{*}

\end{document}